\documentclass[11pt,a4paper]{article}
\usepackage[utf8x]{inputenc}
\usepackage{ucs}
\usepackage{amsmath}
\usepackage{amsfonts}
\usepackage{amsthm}
\usepackage{amssymb}
\usepackage{geometry}
\usepackage{graphicx}
\usepackage{color}
\usepackage{dsfont}
\usepackage{stackrel}
\usepackage[hidelinks]{hyperref}

\newtheorem{theo}{Theorem}
\newtheorem{lem}{Lemma}
\newtheorem{prop}{Proposition}
\newtheorem{coro}{Corollary}
\newtheorem{rem}{Remark}
\newcommand{\R}{\mathbb{R}}
\newcommand{\E}{\mathbb{E}}

\newcommand\blfootnote[1]{%
  \begingroup
  \renewcommand\thefootnote{}\footnote{#1}%
  \addtocounter{footnote}{-1}%
  \endgroup
}

\title{\textcolor{white}{blanc}\\
\vspace{-3cm}Variations on Salem--Zygmund results for\\  random trigonometric polynomials \\
{\normalsize Application to almost sure nodal asymptotics}}

\author{J\"urgen Angst and Guillaume Poly}

\begin{document}
\maketitle
\blfootnote{Univ Rennes, CNRS, IRMAR - UMR 6625, F-35000 Rennes, France.}
\blfootnote{This work was supported by the ANR grant UNIRANDOM, ANR-17-CE40-0008.} 

\par \vspace{-1cm}
\begin{abstract}
On some probability space $(\Omega, \mathcal F, \mathbb P)$ we consider two independent sequences $(a_k)_{k \geq 1}$ and $(b_k)_{k \geq 1}$  of i.i.d. random variables that are centered with unit variance and which admit a moment strictly higher than two. Let us define the associated random trigonometric polynomial  
\[
f_n(t) :=\frac{1}{\sqrt{n}} \sum_{k=1}^n a_k \cos(kt)+b_k \sin(kt), \quad t \in \mathbb R.
\]
In their seminal work, for Rademacher coefficients, Salem and Zygmund showed that $\mathbb P$ almost surely:
\begin{equation*}
\forall t\in\R,\quad \frac{1}{2\pi}\int_{0}^{2\pi} \exp\left(i t f_n(x)\right) dx \xrightarrow[n\to\infty]~e^{-\frac{t^2}{2}}.
\end{equation*}
In other words, if $X$ denotes an independent random variable uniformly distributed over $[0,2\pi]$,  $\mathbb{P}$ almost surely, under the law of $X$, $f_n(X)$ converges in distribution to a standard Gaussian variable. In this paper, we revisit the above Salem--Zygmund result from different perspectives. Namely,
\begin{itemize} 
 \setlength{\itemsep}{0pt}
\item[i)] we establish a possibly sharp convergence rate for some adequate metric via the Stein's method, 
\item[ii)] we prove a functional counterpart of Salem--Zygmund CLT,
\item[iii)] we extend it to more general distributions for $X$,
\item[iv)] we also prove that the convergence actually holds in total variation. 
\end{itemize}
As an application, in the case where the random coefficients have a symmetric distribution and admit a moment of order $4$, we show that, $\mathbb{P}$ almost surely, for any interval $[a,b] \subset [0, 2\pi]$
\[
\frac{\mathcal N(f_n,[a,b])}{n}  \xrightarrow[n \to +\infty]{} \frac{(b-a)}{\pi \sqrt{3}},
\]
where $\mathcal N(f_n,[a,b])$ denotes the number of real zeros of $f_n$ in the interval $[a,b]$. To the best of our knowledge, such an almost sure result is new in the framework of random trigonometric polynomials, even in the  case of Gaussian coefficients.
\end{abstract}
\par \vspace{1cm}
\textcolor{white}{blanc}

{\small
\tableofcontents
}

\section{Introduction and statement of the results}
\subsection{Introduction}
\if{
The study of random trigonometric polynomials and random trigonometric series has a long story which goes back to the pioneering work of Paley, Salem etc.\cite{kahane,kahane2} and it is still a vibrant area of research.
In this article, we are interested in central limit theorems associated with random trigonometric polynomials of large degree, namely we revisit and reinforce the pioneering work of Salem and Zygmund  \cite{salem1954} on the central limit theorem associated to trigonometric polynomials with random signs. 
Let us precise our framework.
}\fi
Let us start by describing the framework of our study and fix the notations. 
We consider a probability space $(\Omega, \mathcal F, \mathbb P)$ which carries two independent sequences $(a_k)_{k \geq 1}$ and $(b_k)_{k \geq 1}$ of independent and identically distributed random variables. A generic element in $\Omega$ will be denoted by $\omega \in \Omega$. We then consider an independent random variable $X$ with distribution $\mathbb P_X$ in $[0,2\pi]$. This can be achieved by considering the product space 
\[
(\Omega \times [0,2\pi], \mathcal F \times \mathcal B([0,2\pi]), \mathbb P \otimes \mathbb P_X),
\]
with $X$ seen as the identity map from $([0,2\pi],\mathcal B([0,2\pi])$ to itself. 
The expectation with respect to $\mathbb P$ and $\mathbb P_X$ will be denoted by $\mathbb E$ and $\mathbb E_X$ respectively. If $Y$ is an independent copy of $X$, which is also independent of the whole sequence $(a_k,b_k)_{k \geq 1}$, we will denote by $\mathbb E_{X,Y}$ the expectation under $\mathbb P_X \otimes \mathbb P_Y$. To the sequence $(a_k,b_k)_{k \geq 1}$, one can naturally associate a sequence of random trigonometric polynomials setting for all $n \geq 1$
\[
f_n(t) := \frac{1}{\sqrt{n}} \sum_{k=1}^n a_k \cos(kt)+b_k \sin(kt), \quad t \in \mathbb R.
\]
The starting point of the article is the following celebrated result of Salem and Zygmund in \cite{salem1954}, which, slightly adapted to our context, reads as follows:
\begin{theo}[Theorem 3.1.1 of \cite{salem1954}]\label{gene-quali-theo}
Suppose that $(a_k,b_k)_{k \geq 1}$ is a sequence of independent and identically distributed random variables that are centered with unit variance and admit a third moment. Let $X$ be an independent random variable that is uniformly distributed over $[0,2\pi]$. Then, $\mathbb P$ almost surely, under $\mathbb P_X$ we have the following convergence in distribution
\begin{equation}\label{eq.Salem}
f_n(X)\xrightarrow[n\to\infty]{\mathrm{law \, under} \,\mathbb{P}_X}\mathcal{N}(0,1),
\end{equation}
in the sense that $\mathbb{P}$ almost surely
\begin{equation}\label{eq.Salemcharac}
\forall t\in\R,\;\; \mathbb E_X\left[  e^{i t f_n(X)}\right] = \frac{1}{2\pi} \int_0^{2\pi} e^{i t f_n(x)}dx \xrightarrow[n \to +\infty]{} e^{-t^2/2}.
\end{equation}
\end{theo}
This result has a long heritage, in particular there is a tremendous literature on its extension to lacunary trigonometric polynomials, with large or bounded gaps, i.e. polynomials of the form 
\[
\frac{1}{\sqrt{n}} \sum_{k=1}^n a_k \cos(n_k t)+b_k \sin(n_k t)
\]
with $n_{k+1}/n_k \geq q>1$ or $1< n_{k+1}-n_k = O(1)$, see e.g. \cite{erdos, berkes,goetze,fukuyama} and the references therein.
We rather focus here on the standard case where $n_k=k$, and let us emphasize that the two problems, non-lacunar or lacunar, are of very different nature. 
Indeed, provided that the subsequence  $(n_k)_{k\ge 1}$ has a sufficient growth rate, one has
\[
\frac{1}{\sqrt{n}} \sum_{k=1}^n  \cos(n_k X)\xrightarrow[n\to\infty]{\textrm{law under}~\mathbb{P}_X}~\mathcal{N}\left(0,\frac{1}{2}\right).
\]
In particular, there is no need to add random weights in front of the cosine in order to catch a Gaussian behavior with respect to $\mathbb{P}_X$. Roughly speaking, this is due to the fine arithmetic properties on the subsequence $(n_k)_{k\ge 1}$ that are guaranteed by the lacunarity. Nevertheless, if no lacunarity is imposed, there is no Gaussian limit behavior in general, as illustrated by the simple counterexample
\[
\frac{1}{\sqrt{n}} \sum_{k=1}^n  \cos(k X) = \frac{1}{\sqrt{n}} \frac{\cos\left(\frac{(n+1)X}{2}\right)\sin\left(\frac{nX}{2}\right)}{\sin\left(\frac{X}{2}\right)}\xrightarrow[n\to\infty]{\mathbb{P}_X\, \textrm{almost surely}}~0.
\]
Let us also emphasize that Theorem \ref{gene-quali-theo} or its lacunar analogues are somewhat reminiscent  to a derandomization phenomenon, since the result holds almost surely in the coefficients $(a_k,b_k)_{k \geq 1}$, and the homogenization follows from the single randomness of the uniform variable $X$. 
Such phenomena are at the heart of the method developed in \cite{bourgain,buckley} to get estimates for nodal observables associated to deterministic eigenfunctions on the torus, starting from the ones associated to random Gaussian eigenfunctions. 
\par
\medskip
As stated in the abstract, our main goal in this article is first, to extend the above Salem--Zygmund Theorem \ref{gene-quali-theo} in different directions. These generalizations are described informally just below and in more details in the next Section \ref{sec.intro.SZ}. Secondly, we show that these variations on Salem--Zygmund results allow to deduce almost sure asymptotics for the empirical measure associated with the real roots of the random trigonometric polynomial $f_n$. In particular, in Theorem \ref{theo.as} which is formally stated in Section \ref{sec.statezero} below, 
we prove under mild conditions on the distribution of $(a_k,b_k)_{k \geq1}$, that $\mathbb P$ almost surely, and for any interval $[a,b] \subset [0, 2\pi]$
\begin{equation} \label{eq.conv-ps}
\frac{\mathcal N(f_n,[a,b])}{n}  \xrightarrow[n \to +\infty]{} \frac{b-a}{\pi\sqrt{3}},
\end{equation}
where $\mathcal N(f_n,[a,b])$ denotes the number of real zeros of $f_n$ in the interval $[a,b]$. To the best of our knowledge, this almost sure asymptotics is new, even in the case of Gaussian coefficients. \par
\medskip
In the case where the random coefficients $(a_k,b_k)_{k \geq 1}$ are not Gaussian, an important additional step in the proof of \eqref{eq.conv-ps} consists in proving some $\log$ integrability of $f_n$. To do so, we rely on the nice recent results of \cite{nishry}. Note that in the case of Gaussian coefficients, this delicate step can be bypassed by elementary means. 
Let us also highlight that, since the limit does not depend on the particular distribution of the random coefficients $(a_k,b_k)_{k \geq 1}$, the above almost sure convergence can be seen as a universality phenomenon. In particular, under an extra moment assumption on the coefficients $(a_k,b_k)$, one immediately recovers the main result of \cite{flasche}, since taking the expectation in \eqref{eq.conv-ps}, and using dominated convergence, we get indeed that
\begin{equation} \label{eq.conv-uni}
 \frac{\mathbb E \left[\mathcal N(f_n,[a,b])\right]}{n}  \xrightarrow[n \to +\infty]{} \frac{b-a}{\pi\sqrt{3}}.
\end{equation}
Before doing it more formally in the next section, let us now say a few words on the kind of extensions of Salem--Zygmund result that are  needed to establish the aforementioned almost sure asymptotics. As a first ingredient, we obtain a quantitative version of Theorem \ref{gene-quali-theo}, i.e. an explicit rate of convergence, by using the so-called Stein's method. The latter is indeed a powerful and versatile method enabling to establish quantitative limit theorems for a great variety of target distributions, see \cite{chen} for a nice introduction to the subject. We moreover prove a functional version of Salem--Zygmund result, showing that, $\mathbb P$ almost surely, under $\mathbb P_X$, the sequence of processes $(f_n(X+t/n))_{t \in [0,2\pi]}$ converges in distribution in the $\mathcal C^1$ topology to a stationary limit Gaussian process with $\sin_c$ covariance function. 
These two first results are key steps in our proof of the almost sure convergence \eqref{eq.conv-ps}.
\par
\medskip
We also obtain more general conditions on $X$ so that \eqref{eq.Salem} holds, and we prove that the convergence  holds not only in distribution but also in total variation. Concretely, we prove that $\mathbb P$ almost surely, $f_n(X)$ has a density under $\mathbb P_X$ and the latter converges in $\mathbb L^1$ to the standard Gaussian density. Although there are not used in the almost sure nodal asymptotics, these two results are of independent interest. 
\par
\medskip
As said above, our initial motivation to establish these variations on Salem--Zygmund convergence \eqref{eq.Salem}  is the study of the nodal set associated to the random trigonometric polynomial $f_n$. The number of real zeros of $f_n$ is the object of a vast literature. For example, in the case of Gaussian coefficients, the expected number of real zeros in a given interval has been investigated in \cite{dunnage,samba,farah}, whereas the variance of this number of zeros was first described in \cite{wigman} and then revisited and generalized in \cite{azais,azais2014clt}. For more general coefficients, some universality results at local and global scales were obtained in a series of recent papers, see for instance \cite{nous,flasche,iksanov2016,nousAMS,MR3846831,Bally2018} and the references therein. 
Quite surprisingly, in the case of real roots of random trigonometric polynomials, the question of the almost sure asymptotics \eqref{eq.conv-ps} has not been tackled until now.
To establish such an almost sure result, a classical and natural strategy would be to get some good estimates on the variance or some higher moment, i.e. for some $p$ large enough, to show that
\begin{equation}\label{eq.summand}
\sum_{n \geq 1} \mathbb E \left[ \left|  \frac{\mathcal N(f_n,[a,b])}{n} - \frac{b-a}{\pi\sqrt{3}}\right|^p \right] <+\infty.
\end{equation}
This is precisely the approach used by Ancona and Letendre in their recent paper \cite{ancona} on the almost sure asymptotics of the real roots of Kostlan algebraic polynomials.
It has also been used for instance in \cite{neu} in the case of nodal set of random spherical harmonics. We refer to Theorem 1 in \cite{zelditch}  for related discussions.
Unfortunately, in the case of random trigonometric polynomials, only the variance, thus the case $p=2$, has been investigated in \cite{wigman}, where the authors showed that the summands in \eqref{eq.summand} are exactly of the order $1/n$, which is naturally not sufficient to conclude to the convergence of the above serie. Such a strategy would thus require at least to handle the case $p \geq 4$, which if doable, would imply quite involved computations.
\par
\medskip
The method we employ here is radically different since we work with the sole randomness induced by the variable $X$, hence in essence, our approach provides almost sure results with respect to the probability $\mathbb P$ associated with the random coefficients. In particular, we do not require any estimate on the variance of the number of zeros.

\subsection{Variations on Salem--Zygmund Theorem}\label{sec.intro.SZ}
Let us now describe more formally and in more details our main results in relation with Salem--Zygmund Theorem \ref{gene-quali-theo}. 
The proofs of Theorems \ref{Theo-Salem-quanti}, \ref{theo.functional}, \ref{SZ-gene}, and \ref{theo.dtv} stated below are the object of the whole Section \ref{sec.proof} of the article.

\subsubsection{A quantitative version of Salem--Zygmund Theorem}
A first direction in which Theorem \ref{gene-quali-theo} can be extended is a quantification of the convergence in distribution for an appropriate choice of metric.
Concretely, let us introduce the $\mathcal C^3$ distance with respect to $X$, that is defined by 
\[
d^X_{\mathcal{C}^3}\left(U,V\right) :=
 \sup_{\substack{\|\phi^{(k)}\|_\infty\le 1\\0\le k\le 3}}\E_X\left[\phi\left(U\right)-\phi\left(V\right)\right],
 \]
whenever $U$ and $V$ are measurable with respect to $\mathcal F \otimes \mathcal B([0,2\pi])$. Then, using the separability of the space of $\mathcal C^3$ functions with bounded derivatives,  the above distance $d^X_{\mathcal{C}^3}\left(U,V\right)$ is a $\mathcal F-$measurable random variable. It actually quantifies the distance between the laws of $U$ and $V$, under $\mathbb P_X$, conditionally on $\mathcal F$. Therefore, it is a convenient tool to measure the rate of convergence in Theorem \ref{gene-quali-theo}. 

\begin{theo}\label{Theo-Salem-quanti}Suppose that $(a_k,b_k)_{k \geq 1}$ is a sequence of independent and identically distributed random variables such that $\E\left[a_1\right]=0,\E\left[a_1^2\right]=1$ and $\E\left[a_1^4\right]<\infty.$ Setting,
\[
C(a_1):=81\sqrt{13+|\mathbb E[a_1^3]|}+8\sqrt{\E\left[a_1^4\right]}+\sqrt{2}+8\E\left[|a_1|^3\right]+24 \E\left[|a_1|\right],
\]
if $G$ is $\sigma(X)-$measurable and if $G\sim\mathcal{N}(0,1)$ under $\mathbb P_X$, then one has
\begin{equation}\label{main-ineq-theo-quanti}
\E\left[ d^X_{\mathcal{C}^3}\left(f_n(X),G\right)\right]\le \frac{C(a_1)}{\sqrt{n}}.
\end{equation}
\end{theo}

The last estimate \eqref{main-ineq-theo-quanti} gives an upper bound for the expected rate of convergence in Theorem \ref{gene-quali-theo}. Using some standard Borel--Cantelli arguments, it is then possible to extract a rate of convergence in the almost sure sense.
\begin{coro}\label{speed-cv-ps}
Under the assumptions of Theorem \ref{Theo-Salem-quanti}, for every $\beta<\frac{1}{6}$, $\mathbb{P}$ almost surely, there exists a constant $C(\omega)>0$ such that for all $n\ge 1$
\begin{equation}\label{speed-ps-ineq}
d_{\mathcal{C}^3}^X\left(f_n(X),G\right)\le \frac{C}{n^\beta}.
\end{equation}
\end{coro}

\begin{rem}\label{Remark-sharp}
Let us emphasize here that the content of Corollary \ref{speed-cv-ps} is not expected to be sharp, as Theorem \ref{Theo-Salem-quanti} indicates that the correct order of convergence is of magnitude $1/\sqrt{n}$, as in the classical Berry--Essen inequality. However, the estimate \eqref{speed-ps-ineq} is sufficient to carry out our proof of the almost sure convergence of empirical measure of roots associated with random trigonometric polynomials. 
\end{rem}
\noindent
The proofs of Theorem \ref{Theo-Salem-quanti} and of its Corollary \ref{speed-cv-ps} are the object of Section \ref{sec.quanti} below. 

\subsubsection{A functional version of Salem--Zygmund Theorem}
Let us introduce the stochastic process $(g_n(t))_{t \in [0,2\pi]}$ defined by 
\[
g_n(t):=f_n\left(X+ \frac{t}{n}\right).
\]
The convergence \eqref{eq.Salem} is naturally equivalent to the fact that $\mathbb P$ almost surely, under $\mathbb P_X$, the sequence $g_n(0)$ converges in distribution to a Gaussian variable. Our next result show that, in fact, the whole process $(g_n(t))_{t \in [0,2\pi]}$ actually converges to an explicit stationary Gaussian process. 
 
\begin{theo}\label{theo.functional}Suppose that $(a_k,b_k)_{k \geq 1}$ is a sequence of independent and identically distributed random variables that are centered with unit variance. Then $\mathbb P$ almost surely, as $n$ goes to infinity, the process $(g_n(t))_{t \in [0,2\pi]}$ converges in distribution in the $\mathcal C^1$ topology, to a stationary Gaussian process $(g_{\infty}(t))_{t \in [0,2\pi]}$ with $\sin_c$ covariance function, i.e.
\[
\mathbb E_X[ g_{\infty}(t) g_{\infty}(s)] = \frac{\sin(t-s)}{t-s}.
\]
\end{theo}

Associated with the continuous mapping Theorem and the continuity of the number of zeros with respect to the $\mathcal C^1$ topology at non-degenerate points, the last Theorem ensures that $\mathbb P$ almost surely, the number of roots of $g_n$ in a compact set, converges in distribution under $\mathbb P_X$, towards its analogue for $g_{\infty}$, see Section \ref{sec.TCLfunc} below.
\par
\medskip
\noindent
The proof of Theorem \ref{theo.functional} is the object of Section \ref{sec.functional} below.

\subsubsection{Salem--Zygmund Theorem for a non-uniform distribution}
\label{sec.introquali}
In the next Section \ref{sec.quali} below, we give an alternative proof of Theorem \ref{gene-quali-theo} which noticeably  differs from the original one. Our strategy mainly exploits the fact that, for all $\xi \in \mathbb R$
\begin{equation}\label{averaging-technique}
\E\left[\left|\E_X\left[e^{i\xi f_n(X)}-e^{-\frac{\xi^2}{2}}\right]\right|^2\right]\to 0
\end{equation} 
sufficiently fast in order to use some Borel--Cantelli argument. In order to do so, if $Y$ denotes an independent copy of $X$, one must show that $\varepsilon_n(X,Y):=\frac{1}{n}\sum_{k=1}^n \cos\left( k (X-Y)\right)$ tends to zero in $\mathbb L^2$ at some adequate speed. In the original case treated by Salem and Zygmund where $X$ and $Y$ are independent uniform variables, we have the exact computation $\mathbb E_{X,Y}[\varepsilon_n(X,Y)^2]=1/2n$ which is sufficient to conclude. 
But this \textit{averaging} strategy applies to more general frameworks. Indeed, the use of Fubini inversion of sums in \eqref{averaging-technique} reveals that a quantified bivariate central convergence of $(f_n(X),f_n(Y))$ as $n\to\infty$ is enough to carry out this strategy. Following the latter method, one can indeed extend Theorem \ref{gene-quali-theo} to more general distributions $\mathbb{P}_X$ over $[0,2\pi]$.

\begin{theo}\label{SZ-gene}
Let $(a_k,b_k)_{k\ge 1}$ be a sequence of independent and identically distributed random variables that are centered with unit variance and which admit a moment of order $\beta\ge3$. Let $X$ be an independent random variable on $[0,2\pi]$ whose Fourier coefficients satisfy
$$\exists \alpha>0,\,\forall k\in\mathbb{Z}/\{0\},\,\left|\widehat{\mathbb{P}_X}(k)\right|\le \frac{C}{|k|^\alpha}.$$
Then,  provided that  $\beta>\frac{2}{\min\left(\alpha,\frac{1}{2}\right)}$, $\mathbb P$ almost surely, under $\mathbb{P}_X$, one has
\[
f_n(X)\xrightarrow[n\to\infty]{\text{law under}~\mathbb{P}_X}~\mathcal{N}(0,1).
\]
\end{theo}

\begin{rem}
The condition imposed on $X$ in the previous Theorem \ref{SZ-gene} is satisfied whenever $X$ admits a density which is Hölder regular. It can also be satisfied for non absolutely continuous distributions like for instance uniform distributions on Cantor sets.
\end{rem}
\noindent
The proof of Theorem \ref{SZ-gene} is the object of Section \ref{sec.quali} below.

\subsubsection{Salem--Zygmund Theorem in total variation}\label{sec.statetv}
Our last extension of Salem--Zygmund Theorem is a total variation version of the convergence in distribution \eqref{eq.Salem}.
As before, for $U$ and $V$ two random variables that are measurable with respect to $\mathcal F \otimes \mathcal B([0,2\pi])$, let us introduce the total variation distance with respect to $\mathbb P_X$
\[
d_{TV}^X\left(U,V\right) :=
 \sup_{||\phi||_\infty \le 1 }\E_X\left[\phi\left(U\right)-\phi\left(V\right)\right].
 \]
As for the $\mathcal C^3$ distance with respect to $X$, the above quantity represents  the conditional total variation distance with respect to $X$, conditionally on $\mathcal F$. It induces a much stronger topology than the convergence in distribution, hence the next Theorem is a substantial reinforcement of Theorem \ref{gene-quali-theo}. 

\begin{theo}\label{theo.dtv}Let $(a_k,b_k)_{k\ge 1}$ be independent and identically distributed random variables that are centered with unit variance and admit a third moment.
Almost surely with respect to the probability $\mathbb P$, if $G$ is $\sigma(X)-$measurable and if $G\sim\mathcal{N}(0,1)$ under $\mathbb P_X$, then as $n$ goes to infinity, we have 
\[
\lim_{n \to +\infty} \mathrm{d}_{TV}^X\left( f_n(X), G\right)=0.
\]
\end{theo} 

\noindent
The proof of Theorem \ref{theo.dtv} is the object of Section \ref{sec.tv} below.

\subsection{Application to almost sure nodal asymptotics} \label{sec.statezero}
Let us now describe how the extensions of Salem--Zygmund result stated above can be used to obtain almost sure results for the number of zeros of random trigonometric polynomials. 
Our main result is the following.
\begin{theo}\label{theo.as}Let us consider a random trigonometric polynomial \[
f_n(t) :=\frac{1}{\sqrt{n}} \sum_{k=1}^n a_k \cos(kt)+b_k \sin(kt), \quad t \in \mathbb R,
\]
where $(a_k)$ and $(b_k)$ are independent sequences of independent and identically distributed random variables with a symmetric distribution, with unit variance and a moment of order four.
Then, $\mathbb P$ almost surely, we have as $n$ goes to infinity
\[
\lim_{n \to +\infty} \frac{\mathcal N(f_n, [0,2\pi]) }{n} = \frac{2}{\sqrt{3}}.
\]
and more generally for any interval $[a,b] \subset [0,2\pi]$
\[
\lim_{n \to +\infty} \frac{\mathcal N(f_n, [a,b]) }{n} = \frac{b-a}{\pi\sqrt{3}}.
\]
\end{theo}

The proof of Theorem  \ref{theo.as} is the object of the whole Section \ref{sec.applizero}. The starting point of the proof is a simple representation of the number of zeros $\mathcal N(f_n, [0,2\pi])$ as an expectation under $\mathbb E_X$, where $X$ is an independent uniform variable, see Section \ref{sec.formula}. Then, the two main ingredients of the proof are Theorem \ref{Theo-Salem-quanti} and Theorem \ref{theo.functional} stated above. Associated with some logarithmic integrability estimates, see Section \ref{sec.logint}, they allow us to deduce some moment estimates in Section \ref{sec.moments}, and to conclude in Section \ref{sec.conclu}.
The next synthetic scheme gives the architecture of the proof.
\begin{figure}[ht]
{\small 
\scalebox{0.715}{
 \begingroup%
  \makeatletter%
  \providecommand\color[2][]{%
    \errmessage{(Inkscape) Color is used for the text in Inkscape, but the package 'color.sty' is not loaded}%
    \renewcommand\color[2][]{}%
  }%
  \providecommand\transparent[1]{%
    \errmessage{(Inkscape) Transparency is used (non-zero) for the text in Inkscape, but the package 'transparent.sty' is not loaded}%
    \renewcommand\transparent[1]{}%
  }%
  \providecommand\rotatebox[2]{#2}%
  \ifx\svgwidth\undefined%
    \setlength{\unitlength}{579.55644531bp}%
    \ifx\svgscale\undefined%
      \relax%
    \else%
      \setlength{\unitlength}{\unitlength * \real{\svgscale}}%
    \fi%
  \else%
    \setlength{\unitlength}{\svgwidth}%
  \fi%
  \global\let\svgwidth\undefined%
  \global\let\svgscale\undefined%
  \makeatother%
  \begin{picture}(1,0.5943816)%
    \put(0.03297963,0.42620222){\color[rgb]{0,0,0}\makebox(0,0)[lb]{\smash{$(g_n(t))_{t \in [0,2\pi]} \xrightarrow[C^1 \; \mathrm{topo.}, \; \mathbb P \ a.s.]{\mathrm{law, \; under} \; \mathbb P_X} (g_{\infty}(t))_{t \in [0, 2\pi]}$\\ }}}%
    \put(0.42063702,0.51807596){\color[rgb]{0,0,0}\makebox(0,0)[lb]{\smash{$\mathbb E\left[ \mathbb E_X \left[ \left| \log(|f_n(X)| \right|^p \right] \right]<\infty$ }}}%
    \put(0.41909085,0.40651523){\color[rgb]{0,0,0}\makebox(0,0)[lb]{\smash{$\mathbb E_X \left[ \left| \log(|f_n(X)| \right|^p \right] =O\left(n^{\theta}\right)$ }}}%
    \put(0,0){\includegraphics[width=\unitlength,page=1]{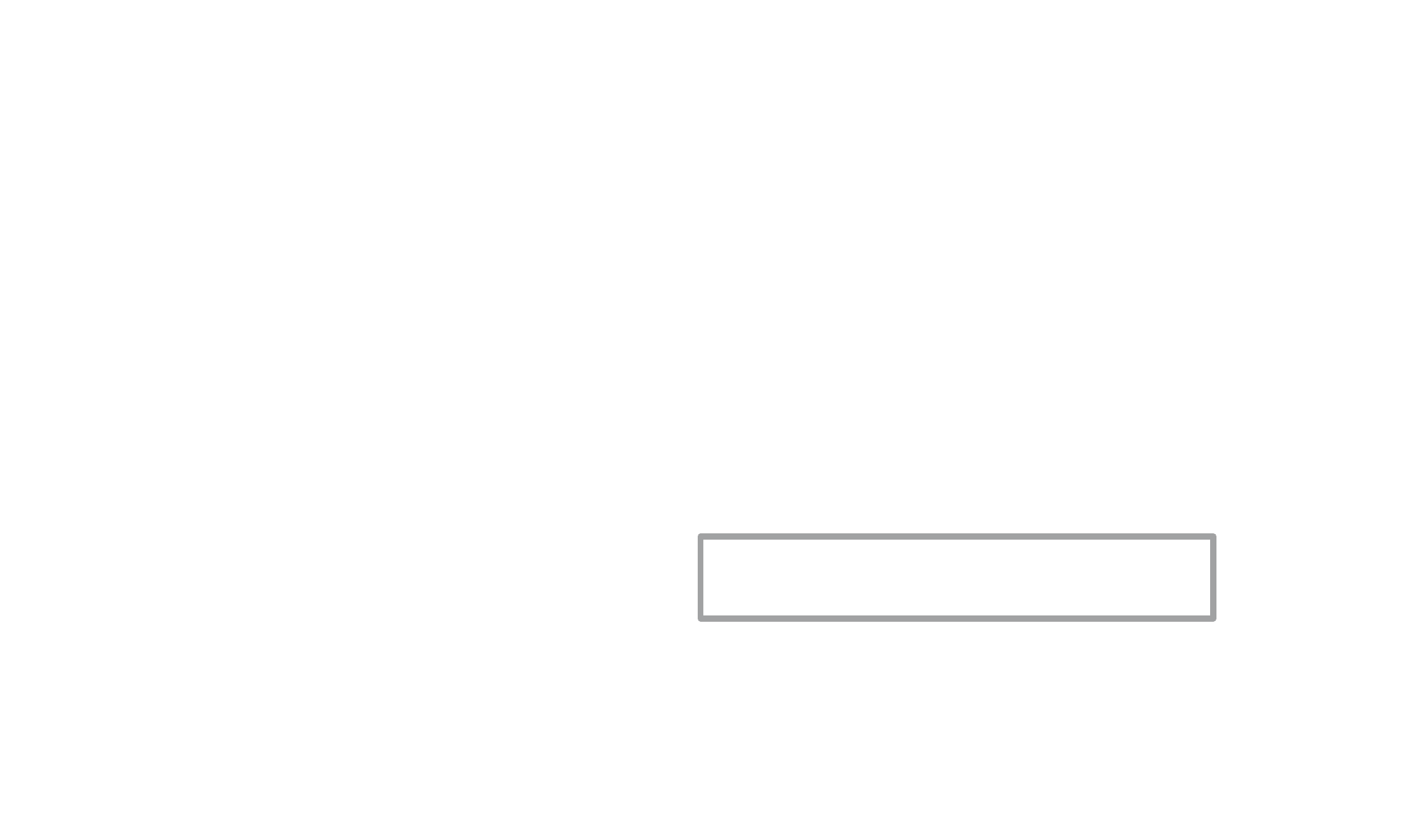}}%
    \put(0.38159929,0.58599156){\color[rgb]{0,0,0}\makebox(0,0)[lb]{\smash{\textcolor{blue}{Logarithmic integrability, Section \ref{sec.logint}}}}}%
    \put(0,0){\includegraphics[width=\unitlength,page=2]{diagramme.pdf}}%
    \put(0.01087241,0.47950618){\color[rgb]{0,0,0}\makebox(0,0)[lb]{\smash{\textcolor{blue}{Functional CLT, Theorem \ref{theo.functional} and Section \ref{sec.TCLfunc}}}}}%
    \put(0.67295512,0.4824641){\color[rgb]{0,0,0}\makebox(0,0)[lb]{\smash{\textcolor{blue}{Quant. estimates., Corollary \ref{speed-cv-ps}}}}}%
    \put(0.10601907,0.22906834){\color[rgb]{0,0,0}\makebox(0,0)[lb]{\smash{\textcolor{blue}{Probabilistic representation, Section \ref{sec.formula}}}}}%
    \put(0,0){\includegraphics[width=\unitlength,page=3]{diagramme.pdf}}%
    \put(0.64288288,0.1250479){\color[rgb]{0,0,0}\makebox(0,0)[lb]{\smash{\textcolor{blue}{Moment estimates, Section \ref{sec.moments}}}}}%
    \put(0.43829291,0.01659061){\color[rgb]{0,0,0}\makebox(0,0)[lb]{\smash{\textcolor{blue}{Conclusion, Section \ref{sec.conclu}}}}}%
    \put(0,0){\includegraphics[width=\unitlength,page=4]{diagramme.pdf}}%
    \put(0.67819333,0.297152){\color[rgb]{0,0,0}\makebox(0,0)[b]{\smash{$\sup_{n \geq 1}\mathbb E_X \left[ \mathcal N(g_n,[0,2\pi])^p \right]<\infty$}}}%
    \put(0.54794372,0.32117575){\color[rgb]{0,0,0}\makebox(0,0)[lb]{\smash{$\forall p\geq1, \; \mathbb P \, a.s.$}}}%
    \put(0.41502142,0.54610412){\color[rgb]{0,0,0}\makebox(0,0)[lb]{\smash{$\forall p>1$}}}%
    \put(0.41595783,0.43816354){\color[rgb]{0,0,0}\makebox(0,0)[lb]{\smash{$\forall p\geq 1, \; \forall \theta>0, \; \mathbb P \; a.s.$}}}%
    \put(0.50824815,0.19463152){\color[rgb]{0,0,0}\makebox(0,0)[lb]{\smash{$\forall p \geq 1, \; \mathbb P \, a.s.$}}}%
    \put(0.67281551,0.1686213){\color[rgb]{0,0,0}\makebox(0,0)[b]{\smash{$\mathbb E_X \left[ \mathcal N(g_n,[0,2\pi])^p \right] \to \mathbb E_X \left[  \mathcal N(g_{\infty}, [0,2\pi])^p\right] $}}}%
    \put(0.25207855,0.18035386){\color[rgb]{0,0,0}\makebox(0,0)[b]{\smash{$\frac{\mathcal N(f_n,[0,2\pi])}{n}=\mathbb E_X \left[  \mathcal N(g_{n}, [0,2\pi])\right] $}}}%
    \put(0.44286485,0.06943159){\color[rgb]{0,0,0}\makebox(0,0)[b]{\smash{$\frac{\mathcal N(f_n,[0,2\pi])}{n}\xrightarrow[\mathbb P\; a.s.]{} \mathbb E_X \left[  \mathcal N(g_{\infty}, [0,2\pi])\right]=\frac{2}{\sqrt{3}}$}}}%
    \put(0.03839451,0.31104271){\color[rgb]{0,0,0}\makebox(0,0)[lb]{\smash{$\mathcal N(g_n, [0,2\pi])\xrightarrow[\mathbb P \; a.s.]{\mathrm{law, \; under}\; \mathbb P_X} \mathcal N(g_{\infty}, [0,2\pi])$}}}%
    \put(0,0){\includegraphics[width=\unitlength,page=5]{diagramme.pdf}}%
    \put(0.70945463,0.42338558){\color[rgb]{0,0,0}\makebox(0,0)[lb]{\smash{$d_{C^3}^X(f_n(X),G)\stackrel{\mathbb P \, a.s.}{=}O(n^{-\alpha})$}}}%
    \put(0,0){\includegraphics[width=\unitlength,page=6]{diagramme.pdf}}%
  \end{picture}%
\endgroup%
}}
\caption{Synthetic plan of the proof of Theorem \ref{theo.as}}
\end{figure}

\section{Proofs of the variations on Salem--Zygmund Theorem}\label{sec.proof}

In this section, we give the proofs of the variations on Salem--Zygmund Theorem, namely our main Theorems 2 to 5 stated in Section \ref{sec.intro.SZ} above.

\subsection{A quantitative Salem--Zygmund Theorem via Stein's method} \label{sec.quanti}
We use the powerful method of Stein to provide a proof of Theorem \ref{Theo-Salem-quanti}, i.e. to quantify the rate of convergence for some appropriate metric in the Salem--Zygmund central limit Theorem. This quantitative bound will play a central role in the proof of the almost sure asymptotics of the number of real zeros of random trigonometric polynomials.

\begin{proof}[Proof of Theorem \ref{Theo-Salem-quanti}.]
The proof will be divided into several steps, directly inspired from the celebrate proof of the Central Limit Theorem which is due to Charles Stein. \par

\medskip
\noindent
\underline{Step 1: \textit{using the Stein equation}}
\medskip

Let $h\in\mathcal{C}^3_b(\R)$ which is centered with respect to the Gaussian distribution and which satisfies $\max\left(\|h\|_\infty,\|h'\|_\infty,\|h''\|_\infty,\|h'''\|_\infty\right)\le 1$. The Stein equation is

\begin{equation}\label{Stein-eq}
\forall x\in\R,\,\phi'(x)-x\phi(x)=h(x).
\end{equation}

The following Theorem is at the heart of the Stein methodology for Gaussian approximation, see Chapter 2.2 of \cite{chen}. 
\begin{theo}\label{Stein}
Let $h\in\mathcal{C}^1_b(\R)$ such that $\max\left(\|h\|_\infty,\|h'\|_\infty\right)\le1$, then there exists a unique solution to \eqref{Stein-eq} which satisfies:
\begin{equation}\label{Stein-bounds}
\|\phi_h\|_\infty\le 2,\,\|\phi_h'\|_\infty\le 4,\,\|\phi_h''\|_\infty\le 2.
\end{equation}
\end{theo}

At the end of the proof, we shall need the following technical corollary which proceeds directly from Theorem \ref{Stein}.

\begin{coro}\label{Stein-bounds-technical}
Under the stronger assumption $\max\left(\|h\|_\infty,\|h'\|_\infty,\|h''\|_\infty,\|h'''\|_\infty\right) \leq 1$ one gets
\begin{equation}\label{tech-stein}
\|x\phi_h\|_\infty\le 5 \quad \&\quad \|\phi^{(3)}\|_\infty \le 18 \quad \& \quad \|x \phi_h^{(3)}\|_\infty \le 43.
\end{equation}
\end{coro}

\begin{proof}[Proof of the Corollary \ref{Stein-bounds-technical}]
Taking two derivatives of the equation \eqref{Stein-eq} gives the equation $\phi_h^{(3)}-x\phi_h''=h''+2\phi_h'$ which is another Stein equation with left hand side $h''+2\phi_h'$. Relying on \eqref{Stein-bounds} and the assumptions on $h$ we claim that $\max(\|h''+2\phi_h'\|_\infty,\|h''+2\phi_h''\|_\infty)\le 9$. This ensures, again by Theorem \ref{Stein} , that $\|\phi_h^{(3)}\|_\infty\le 18$ and $\|\phi_h^{(4)}\|_\infty\le 36$. Finally, taking again a derivative of the Stein equation leads to $\phi_h^{(4)}-x\phi_h^{(3)}=h'''+3\phi^{(2)}$ and the previous bounds guarantee that $\|x\phi_h^{(3)}\|_\infty\le 43.$
\end{proof}

We shall denote by $\mathcal{E}$ the set of functions which satisfy the conditions \eqref{Stein-bounds} and \eqref{tech-stein}.
Based on the combination of \eqref{Stein-eq},\eqref{Stein-bounds} and \eqref{tech-stein}, any function $h$ that is centered for the standard Gaussian distribution and whose three first derivatives are bounded by one can be written $\phi_h'-x \phi_h$ for some $\phi\in\mathcal{E}$. We have thus

\begin{equation}\label{Stein-eq1}
\begin{array}{ll}
  \displaystyle{d^{X}_{\mathcal{C}^3}\left(\frac{S_n(X)}{\sqrt{n}},G\right) }\le   \displaystyle{\sup_{\phi\in\mathcal{E}} \left|\mathbb{E}_X\left[\phi'\left(\frac{S_n(X)}{\sqrt{n}}\right)-\frac{S_n(X)}{\sqrt{n}}\phi\left(\frac{S_n(X)}{\sqrt{n}}\right)\right]\right|}.
 \end{array}
\end{equation}

We fix in the forthcoming estimates some function $\phi$ which obeys to the bounds given in Equation \eqref{Stein-bounds}. For convenience, we also set $R_k(X):=a_k \cos(k X)+b_k \sin( k X)$ and $S_n^k(X):=S_n(X)-R_k(X)$. Using Taylor expansions and denoting by $U$ an independent random variable uniformly distributed over $[0,1]$, we may write
\begin{eqnarray*}
&&\mathbb{E}_X\left[\frac{S_n(X)}{\sqrt{n}}\phi\left(\frac{S_n(X)}{\sqrt{n}}\right)\right]=\frac{1}{\sqrt{n}}\sum_{k=1}^n\mathbb{E}_X\left[R_k(X)~\phi\left(\frac{S_n(X)}{\sqrt{n}}\right)\right]\\
&=&\underbrace{\frac{1}{\sqrt{n}}\sum_{k=1}^n\mathbb{E}_X\left[R_k(X)~\phi\left(\frac{S_n^k(X)}{\sqrt{n}}\right)\right]}_{:=A_n}
+\underbrace{\frac{1}{n}\sum_{k=1}^n\mathbb{E}_X\left[R_k(X)^2~\phi'\left(\frac{S_n^k(X)}{\sqrt{n}}\right)\right]}_{:=B_n}\\
&+&\underbrace{\frac{1}{n\sqrt{n}}\sum_{k=1}^n\mathbb{E}_X\left[R_k(X)^3~\mathbb{E}_U\left[\phi''\left(U\frac{S_n^k(X)}{\sqrt{n}}+(1-U)\frac{S_n(X)}{\sqrt{n}}\right)\right]\right]}_{:=C_n}.
\end{eqnarray*}
We shall now estimate each of the terms $A_n$, $B_n$ and $C_n$ separately. First, the term $C_n$ can be bounded in the following way
\begin{eqnarray*}
|C_n|&=&\left|\frac{1}{n\sqrt{n}}\sum_{k=1}^n\mathbb{E}_X\left[R_k(X)^3~\mathbb{E}_U\left[\phi''\left(U\frac{S_n^k(X)}{\sqrt{n}}+(1-U)\frac{S_n(X)}{\sqrt{n}}\right)\right]\right]\right|\\
&\le&\frac{\|\phi''\|_\infty}{\sqrt{n}}~~\frac{1}{n}\sum_{k=1}^n \left(|a_k|+|b_k|\right)^3\\
& \leq& \frac{2}{n \sqrt{n}} ~~\sum_{k=1}^n \left(|a_k|+|b_k|\right)^3.
\end{eqnarray*}

\medskip
\noindent
\underline{Step 2: \textit{bounding the term $B_n$}}\par
\medskip
\noindent
We can write the following decomposition:
\[
\begin{array}{ll}
B_n & =\displaystyle{\frac{1}{n}\sum_{k=1}^n\mathbb{E}_X\left[R_k(X)^2~\phi'\left(\frac{S_n^k(X)}{\sqrt{n}}\right)\right]=\mathbb{E}_X\left[\frac{1}{n}\sum_{k=1}^n R_k(X)^2~\phi'\left(\frac{S_n(X)}{\sqrt{n}}\right)\right]}\\
\\
&\displaystyle{- \mathbb{E}_X\left[\frac{1}{n\sqrt{n}}\sum_{k=1}^n R_k(X)^3~\mathbb{E}_U\left[\phi''\left(U\frac{S_n(X)}{\sqrt{n}}+(1-U)\frac{S_n^k(X)}{\sqrt{n}}\right)\right]\right]}\\
\end{array}
\]
so that 
\[
\begin{array}{ll}
\displaystyle{B_n -\mathbb{E}_X\left[\phi'\left(\frac{S_n(X)}{\sqrt{n}}\right)\right]}  =
\displaystyle{\underbrace{\mathbb{E}_X\left[\left(\frac{1}{n}\sum_{k=1}^n R_k(X)^2-1\right)~\phi'\left(\frac{S_n(X)}{\sqrt{n}}\right)\right]}_{:=B_{n,1}}}\\
\\
\displaystyle{-\underbrace{\mathbb{E}_X\left[\frac{1}{n\sqrt{n}}\sum_{k=1}^n R_k(X)^3~\mathbb{E}_U\left[\phi''\left(U\frac{S_n(X)}{\sqrt{n}}+(1-U)\frac{S_n^k(X)}{\sqrt{n}}\right)\right]\right]}_{:=B_{n,2}}}.
\end{array}
\]
\noindent
Using Stein's bounds given in \eqref{Stein-bounds} we then get
\begin{eqnarray*}
|B_{n,2}|&\le& \frac{2}{n \sqrt{n}}~~\sum_{k=1}^n \left(|a_k|+|b_k|\right)^3,\\
|B_{n,1}|&\le& 4~\mathbb{E}_X\left[\left|\frac{1}{n}\sum_{k=1}^n R_k(X)^2-1\right|\right] \leq  4 \mathbb{E}_X\left[\left|\frac{1}{n}\sum_{k=1}^n a_k^2 \cos^2(k X)-\frac{1}{2}\right|\right]\\
&+&4 \mathbb{E}_X\left[\left|\frac{1}{n}\sum_{k=1}^n b_k^2 \sin^2(k X)-\frac{1}{2}\right|\right]+ 4 \mathbb{E}_X\left[\left|\frac{1}{n}\sum_{k=1}^n a_k b_k \cos(k X) \sin (k X)\right|\right].
\end{eqnarray*}
Then, using in the one hand the formula $\cos(2x)=2\cos^2(x)-1$ and Cauchy--Schwarz inequality on the other hand, we may infer
\begin{eqnarray*}
 \mathbb{E}_X\left[\left|\frac{1}{n}\sum_{k=1}^n a_k^2 \cos^2(k X)-\frac{1}{2}\right|\right]&\le& \frac{1}{2}\, \mathbb{E}_X\left[\left|\frac{1}{n}\sum_{k=1}^n a_k^2 \cos(2 k X)\right|\right]
 +\frac{1}{2}\left|\frac{1}{n}\sum_{k=1}^n \left(a_k^2 -1\right) \right|\\
&\le& \frac{1}{2} \sqrt{\frac{1}{2 n^2}\sum_{k=1}^n a_k^4}+\frac{1}{2}\left|\frac{1}{n}\sum_{k=1}^n \left(a_k^2 -1\right) \right|.
\end{eqnarray*}
Exactly in the same way, we obtain
\begin{eqnarray*}
\mathbb{E}_X\left[\left|\frac{1}{n}\sum_{k=1}^n b_k^2 \sin^2(k X)-\frac{1}{2}\right|\right]&\le&\frac{1}{2} \sqrt{\frac{1}{2 n^2}\sum_{k=1}^n b_k^4}+\frac{1}{2}\left|\frac{1}{n}\sum_{k=1}^n \left(b_k^2 -1\right) \right|,
\end{eqnarray*}
as well as
\begin{eqnarray*}
\mathbb{E}_X\left[\left|\frac{1}{n}\sum_{k=1}^n a_k b_k \cos(k X) \sin (k X)\right|\right]&\le& \frac{1}{2}\mathbb{E}_X\left[\left|\frac{1}{n}\sum_{k=1}^n a_k b_k \sin (2 k X)\right|\right]\\
&\le&\frac{1}{2} \sqrt{ \frac{1}{2 n^2}\sum_{k=1}^n a_k^2 b_k^2}.
\end{eqnarray*}
\noindent
Gathering all these bounds, we finally obtain
\begin{equation}\label{bound-B_n-1}
\left| B_n -\mathbb{E}_X\left[\phi'\left(\frac{S_n(X)}{\sqrt{n}}\right)\right]\right| \le \frac{\lambda_n}{\sqrt{n}}
\end{equation}
with 
\begin{equation}
\begin{array}{ll}\label{bound-B_n-2}
\lambda_n&:= \displaystyle{\frac{2}{n}\sum_{k=1}^n\left(|a_k|+|b_k|\right)^3 + 2 \sqrt{\frac{1}{2 n} \sum_{k=1}^n a_k^4}+2\left|\frac{1}{\sqrt{n}}\sum_{k=1}^n \left(a_k^2 -1\right) \right|}\\
&  \displaystyle{+2 \sqrt{\frac{1}{2 n}\sum_{k=1}^n b_k^4}+2\left|\frac{1}{\sqrt{n}}\sum_{k=1}^n \left(b_k^2 -1\right) \right|+ 2\sqrt{ \frac{1}{2 n}\sum_{k=1}^n a_k^2 b_k^2}}.
 \end{array}
\end{equation}
\medskip
\medskip
\underline{Step 3: \textit{bounding the term $A_n$ in the  particular case} $\phi(\cdot)=\exp\left(i \xi \cdot\right)$}\par
\medskip
Bounding the term $A_n$ term is arguably the most difficult part of the proof. Note that for the moment, the estimates obtained for $B_n$ and $C_n$ are almost sure with respect to $\mathbb P$. We were not able to give such an almost sure estimate for $A_n$, but it is possible to estimate its expectation under $\mathbb P$. In order to capture the correct order of convergence, we will first handle the case $\phi(\cdot)=\exp\left(i \xi \cdot\right)$ with $\xi\in\mathbb{R}$.
To this end, let us define 
\[
\tilde{A}_n(\xi):=\mathbb{E}\left(\left|\frac{1}{\sqrt{n}}\sum_{k=1}^n\mathbb{E}_X\left[R_k(X)\exp\left(i \xi  S_n^{k}(X)\right)\right]\right|^2\right).
\]
The next Lemma shows that $\tilde{A}_n(\xi)$ is of the order $1/n$. 

\begin{lem} For all $ \xi \in \mathbb{R}$, we have 
\begin{equation}\label{Antilde-final}
\tilde{A}_n(\xi)\le \left(13+|\mathbb E [a_1^3]|\right)\frac{|\xi|^4+|\xi|^3+|\xi|^2+1}{n}.
\end{equation}\label{lem.Antilde}
\end{lem}

The proof of Lemma \ref{lem.Antilde} is quite technical, in order to facilitate the global reading of the paper, it is postponed in Section \ref{app.lem1} of the appendix.

\medskip
\noindent
\underline{Step 4: \textit{bounding the term $A_n$}}\par
\medskip
Let $\phi$ the unique solution of the Stein equation $\phi_h'-x\phi_h=h$ where it is assumed that $\max\left(\|h'\|_\infty,\|h''\|_\infty,\|h'''\|_\infty\right)\le 1$. By the Step 1 of the proof we know that $\phi\in\mathcal{E}$ which means that estimates \eqref{Stein-bounds} and \eqref{tech-stein} are fulfilled. In particular, using Plancherel Theorem we have
\[
\begin{array}{ll}
\displaystyle{\int_{\mathbb{R}} \left|\hat{\phi}(\xi)\right|^2 (1+|\xi|^3)^2 d\xi} &\le \displaystyle{2 \int_{\mathbb{R}} \left|\hat{\phi}(\xi)\right|^2 (1+|\xi|^6) d\xi}\\
\\
&=\displaystyle{ 4\pi \int_{\R} \left(|\phi(x)|^2+|\phi^{(3)}(x)|^2\right)dx}\\
\\
&\le  \displaystyle{4\pi\left( \|\phi\|_\infty+\|\phi^{(3)}\|_\infty^2\right)+ 4\pi \int_{|x|>1}\left(|\phi(x)|^2+|\phi^{(3)}(x)|^2\right)dx}\\
\\
&\displaystyle{\stackrel{\eqref{tech-stein}}{\le} 4\pi(4+324)+8\pi\int_{1}^\infty \left(\frac{43^2+5^2}{x^2}\right)dx }\\
\\
&= \displaystyle{4\pi(4+324+2\times 43^2+2\times 5^2)=16304 \pi}.
\end{array}
\]
Using Fourier inversion Theorem, we then get
\[
\begin{array}{ll}
\left|A_n\right| & \displaystyle{=\left|\frac{1}{\sqrt{n}}\sum_{k=1}^n\mathbb{E}_X\left[R_k(X)~\phi\left(\frac{S_n^k(X)}{\sqrt{n}}\right)\right]\right|}\\
&\le \displaystyle{ \frac{1}{2\pi} \int_{\R}|\hat{\phi}(\xi)| \left|\frac{1}{\sqrt{n}}\sum_{k=1}^n\mathbb{E}_X\left[R_k(X)~\exp\left(i \xi \frac{S_n^k(X)}{\sqrt{n}}\right)\right]\right| d\xi}\\
&\le \displaystyle{\frac{1}{2\pi}\int_{\R} |\hat{\phi}(\xi)| (1+|\xi|^3)\frac{1}{1+|\xi|^3}\left|\frac{1}{\sqrt{n}}\sum_{k=1}^n\mathbb{E}_X\left[R_k(X)~\exp\left(i \xi \frac{S_n^k(X)}{\sqrt{n}}\right)\right]\right| d\xi}.
\end{array}
\]
Taking the expectation with respect to $\mathbb{P}$, the Cauchy--Schwarz inequality associated with the above Lemma \ref{lem.Antilde} gives us
\[
\begin{array}{ll}
\E\left[|A_n|^2\right] & \displaystyle{\le \frac{1}{4\pi^2}\int_{\R}|\hat{\phi}(\xi)|^2 (1+|\xi|^3)^2 d\xi \int_{\R} \frac{\tilde{A}_n(\xi)}{(1+|\xi|^3)^2}  d\xi}\\
\\
&\le \displaystyle{ \frac{4076}{\pi}\int_{\R} \frac{1}{(1+|\xi|^3)^2} \left(13+|\mathbb E[a_1^3]|\right)\frac{|\xi|^4+|\xi|^3+|\xi|^2+1}{n} d\xi}\\
\\
&\le (13+|\mathbb E[a_1^3]|) \frac{20380}{n \pi},
\end{array}
\]
where we have used in the last estimate that
$$\int_{\R}\frac{|\xi|^4+|\xi|^3+|\xi|^2+1}{(1+|\xi|^3)^2}d\xi = \frac{2}{27}\left(9+10\sqrt{3}\pi\right)\approx 4.7 \le 5.$$

\medskip
\noindent
\underline{Step 5: \textit{synthesis}}\par
\medskip
The compilation of Step 1 and Step 2, give the following upper bound
\begin{equation}\label{upper-bound1}
\displaystyle{d^{X}_{\mathcal{C}^3}\left(\frac{S_n(X)}{\sqrt{n}},G\right) }\le  |A_n|+|C_n|+\frac{\lambda_n}{\sqrt{n}}.
\end{equation}
Taking the expectation with respect to $\mathbb{P}$ leads to
\[
\begin{array}{ll}
\displaystyle{\E\left[\displaystyle{d^{X}_{\mathcal{C}^3}\left(\frac{S_n(X)}{\sqrt{n}},G\right) }\right] }& \displaystyle{\le \E\left[|A_n|\right]+\E\left[|C_n|\right]+\frac{\E\left[\lambda_n\right]}{\sqrt{n}}}\\
\\
&\le \displaystyle{\underbrace{\sqrt{(13+|\mathbb E[a_1^3]|) \frac{20380}{n \pi}}}_{\text{Cauchy--Schwarz+end of Step 4}}+\underbrace{\frac{2}{\sqrt{n}}\E\left[\left(|a_1|+|b_1|^3\right)\right]}_{\text{End of Step 1}}}\\
\\
&+\displaystyle{ \underbrace{\left(8\sqrt{\mathbb E[a_1^4]}+\sqrt{2}+2\E\left[\left(|a_1|+|b_1|\right)^3\right]\right)\frac{1}{\sqrt{n}}}_{\text{Taking expectation in} \,\eqref{bound-B_n-2}}}.
\end{array}
\]
As $\sqrt{\frac{20380}{\pi}}\approx 80.54$ one can simplify a bit and get
\begin{equation}\label{upper-bound2}
\E\left[\displaystyle{d^{X}_{\mathcal{C}^3}\left(\frac{S_n(X)}{\sqrt{n}},G\right) }\right]\le C(a_1) \frac{1}{\sqrt{n}},
\end{equation}
with
$$C(a_1):=81\sqrt{13+|\mathbb E[a_1^3]|}+8\sqrt{\mathbb E[a_1^4]}+\sqrt{2}+8\E\left[|a_1|^3\right]+24 \E\left[|a_1|\right].$$
\end{proof} 
\noindent
Let us now detail how Corollary \ref{speed-cv-ps} can be deduced from Theorem \ref{Theo-Salem-quanti}.

\begin{proof}[Proof of Corollary \ref{speed-cv-ps}.] 
Let $\beta<\frac{1}{6}$. Based on the inequality \eqref{main-ineq-theo-quanti}, we obtain that

$$\sum_{n=1}^\infty n^{3\beta} \E\left[d^X_{\mathcal{C}^3}\left(f_{n^3}(X),G\right)\right]\le \sum_{n\ge 1}\frac{C(a_1)}{n^{\frac{3}{2}-3\beta}}<\infty,\; \; \text{since}\;\; \frac{3}{2}-3\beta>1.$$
Using Borel--Cantelli Lemma, $\mathbb{P}$ almost surely, there exists a constant $C(\omega)>0$, which may change from line to line, such that
$$d^X_{\mathcal{C}^3}\left(f_{n^3}(X),G\right)\le \frac{C}{n^{3\beta}}.$$
On the other hand, for any integer $m\ge 1$ one may find $n\ge 1$ such that $n^3\le m\le (n+1)^3$. 
Setting $\Delta_{m,n^3}:=\E_X\left[\left|f_m(X)-f_{n^3}(X)\right|^2\right]$, one then deduces that
\[
\begin{array}{ll}
\Delta_{m,n^3} &\le \displaystyle{2 \left(\sqrt{\frac{n^3}{m}}-1\right)^2 \underbrace{\E_X\left[f_{n^3}(X)^2\right]}_{=1}+\frac{2}{n^3}\E_X\left[\left|\sum_{k=n^3+1}^m a_k\cos(k X)+b_k\sin(k X)\right|^2\right]}\\
&\displaystyle{\le\underbrace{2 \left(\sqrt{\frac{n^3}{(n+1)^3}}-1\right)^2+2\frac{(n+1)^3-n^3}{n^3}}_{=O\left(\frac{1}{n}\right)} = O\left(\frac{1}{m^{\frac{1}{3}}}\right)}.
\end{array}
\]
Gathering the previous estimates provides the desired conclusion since
\[
\begin{array}{ll}
\displaystyle{d^X_{\mathcal{C}^3}\left(f_{m}(X),G\right) }& \displaystyle{\le d^X_{\mathcal{C}^3}\left(f_{n^3}(X),G\right)+d^X_{\mathcal{C}^3}\left(f_{m}(X),f_{n^3}\right)}\\
\\
&\le\displaystyle{ \frac{C}{n^{3\beta}}+\sqrt{\E_X\left[\left|f_m(X)-f_{n^3}(X)\right|^2\right]}}\\
\\
&\le\displaystyle{ \frac{C}{m^\beta}+O\left(\frac{1}{m^{\frac{1}{6}}}\right)=O\left(\frac{1}{m^{\beta}}\right).}
\end{array}
\]
\end{proof}

\subsection{Functional Central Limit Theorem}\label{sec.functional}
In this section, we give the proof of Theorem \ref{theo.functional}. The convergence in law of the sequence of processes $(g_n(t))_{t \in [0,2\pi]}$ classically splits into two parts: on the one hand the convergence of finite dimensional marginals, and on the other hand a suitable tightness argument associated with the $\mathcal C^1$ topology. 
%\subsubsection{Convergence of finite dimensional marginals}
Let us first establish the convergence of the finite marginals of the process $(g_n(t))_{t \in [0,2\pi]}$.
\begin{prop}\label{prop.marginal}
Almost surely with respect to the probability $\mathbb P$, as $n$ goes to infinity, the finite marginals of the process $(g_n(t))_{t \in [0,2\pi]}$ converge to the ones of a Gaussian process $(g_{\infty}(t)_{t \in [0,2\pi]}$ with $\sin_c$ covariance function, i.e.
\[
\mathbb E_X[ g_{\infty}(t) g_{\infty}(s)] = \frac{\sin(t-s)}{t-s}.
\]
\end{prop}

\begin{proof}
Let us fix an integer $p \geq 1$, $t=(t_1, \ldots, t_p) \in \mathbb [0,2\pi]^p$ with $t_j \neq t_j$ for $i \neq j$ and $\lambda=(\lambda_1, \ldots, \lambda_p) \in \mathbb R^p$. We set $||\lambda||_1:=\sum_{j=1}^p |\lambda_j|$. We then define 
\[
Q(t,\lambda):= \sum_{i,j} \lambda_i \lambda_j \frac{\sin(t_i-t_j)}{t_i -t_j},
\]
and for $n \geq 1$, the associated Riemann sum
\[
Q_n(t,\lambda):=\sum_{i,j=1}^p \lambda_i \lambda_j \frac{1}{n}\sum_{k=1}^n \cos \left( \frac{k(t_i-t_j)}{n} \right).
\] 
Naturally, since the cosine function has a bounded derivative, using standard comparison results between Riemann sums and their limits, we have 
\[
\left| Q(t,\lambda) - Q_n(t,\lambda)\right| = O(1/n).
\]
Therefore, if we want to establish that, $\mathbb P$ almost surely, as $n$ goes to infinity 
\[
\Psi_n(t,\lambda):=\mathbb E_X \left[ e^{i \sum_{j=1}^p \lambda_j g_n(t_j)}\right] \longrightarrow e^{-Q(t,\lambda)/2},
\]
it is sufficient to show that $\mathbb P$ almost surely
\[
\lim_{n \to \infty} \left| \Psi_n(t,\lambda) -e^{-Q_n(t,\lambda)/2} \right|=0.
\]
We have 
\[
\sum_{j=1}^p \lambda_j g_n(t_j) = \frac{1}{\sqrt{n}} \sum_{k=1}^n a_k \alpha_{k,n}(X)+b_k \beta_{k,n}(X),
\]
where
\[
\alpha_{k,n}(X):=\sum_{j=1}^p \lambda_j \cos \left( k X+\frac{k t_j}{n}\right), \quad \beta_{k,n}(X):=\sum_{j=1}^p \lambda_j \sin \left( k X+\frac{k t_j}{n}\right).
\]
A straightforward computation then yields the following ``orthogonality'' relations, for all $1\leq k \leq n$ 
and $1\leq \ell \leq m$ \begin{equation}\label{eq.ortho1}
\mathbb E_X \left[ \alpha_{k,n}(X) \alpha_{\ell,m}(X)\right] =  \mathbb E_X \left[ \beta_{k,n}(X) \beta_{\ell,m}(X)\right] = \delta_{k \ell}\times \frac{1}{2} \sum_{i,j=1}^p \lambda_i \lambda_j \cos\left( \frac{k t_i}{n}-\frac{k t_j}{m}\right),
\end{equation} 
and 
\begin{equation}\label{eq.ortho2}
\mathbb E_X \left[ \alpha_{k,n}(X) \beta_{\ell,m}(X)\right] =\delta_{k \ell}\times \frac{1}{2} \sum_{i,j=1}^p \lambda_i \lambda_j \sin\left( \frac{k t_i}{n}-\frac{k t_j}{m}\right) ,
\end{equation} 
so that by symmetry
\begin{equation}\label{eq.ortho3}
\mathbb E_X \left[ \alpha_{k,n}(X) \beta_{\ell,n}(X)\right] = 0.
\end{equation} 
In particular, we have 
\begin{equation}\label{eq.bound}
\mathbb E_X \left[ \alpha_{k,n}(X)^2\right] =  \mathbb E_X \left[ \beta_{k,n}(X)^2\right] \leq  \frac{1}{2}  \times ||\lambda||_1^2.
\end{equation} 
Set 
\[
\Delta_n:=\mathbb E \left[\left|\Psi_n(t,\lambda) - e^{-Q_n(t,\lambda)/2} \right|^2\right].
\]
By the Fubini inversion theorem, we have then for $Y$ an independent copy of $X$
\[
\Delta_n=\mathbb E_{X,Y} \left[ A_n \right]- e^{-Q_n(t,\lambda)/2} \times \mathbb E_X \left[ B_n +\overline{B}_n \right] + e^{-Q_n(t,\lambda)},
\]
where we have set
\[
\begin{array}{ll}
A_n & :=\mathbb E\left[ e^{i \frac{1}{\sqrt{n}}\sum_{k=1}^n a_k(\alpha_{k,n}(X) - \alpha_{k,n}(Y)) + b_{k}  (\beta_{k,n}(X) - \beta_{k,n}(Y))} \right],  \\
\\
B_n & :=\mathbb E\left[ e^{i \frac{1}{\sqrt{n}}\sum_{k=1}^n a_k \alpha_{k,n}(X) + b_{k}  \beta_{k,n}(X) } \right].
\end{array}
\]
By the classical comparison results for sum of independent variables with their Gaussian analogues, we can infer that if $(\widetilde{a}_k)_{k \geq 1}$ and $(\widetilde{b}_k)_{k \geq 1}$ are independent sequences of i.i.d. standard Gaussian variables, then uniformly in $X,Y$
\[
\left| A_n - \widetilde{A}_n\right| \leq \frac{C}{\sqrt{n}}, \quad \left| B_n - \widetilde{B}_n \right| \leq \frac{C}{\sqrt{n}},
\] 
where 
\[
\begin{array}{ll}
\widetilde{A}_n & :=\mathbb E\left[ e^{i \frac{1}{\sqrt{n}}\sum_{k=1}^n \widetilde{a}_k(\alpha_{k,n}(X) - \alpha_{k,n}(Y)) + \widetilde{b}_{k}  (\beta_{k,n}(X) - \beta_{k,n}(Y))} \right],  \\
\\
\widetilde{B}_n & := \mathbb E\left[ e^{i \frac{1}{\sqrt{n}}\sum_{k=1}^n \widetilde{a}_k(\alpha_{k,n}(X) + \widetilde{b}_{k}  (\beta_{k,n}(X) } \right].
\end{array}
\]
But since the variables are now independent and Gaussian, the above expectations are explicit. Namely, if we set 
\[
\mathcal E_n(t,\lambda):=\sum_{i,j} \lambda_i \lambda_j \left( \frac{1}{n} \sum_{k=1}^n \cos \left(k(X-Y)+ \frac{k(t_i-t_j)}{n}  \right)\right),
\]
after simplification we have simply
\[
\widetilde{A}_n =\displaystyle{\exp \left(- Q_n(t,\lambda) + \mathcal E_n(t,\lambda) \right)}, 
\quad \widetilde{B}_n  =\displaystyle{\exp\left( -\frac{1}{2} Q_n(t,\lambda)\right)}.
\]
As a consequence, we have 
\[
\begin{array}{ll}
\Delta_n & =\mathbb E_{X,Y} \left[ \widetilde{A}_n \right]- 2 e^{-Q_n(t,\lambda)/2} \times \mathbb E_X \left[ \widetilde{B}_n \right] + e^{-Q_n(t,\lambda)} + O\left( \frac{1}{\sqrt{n}}\right) ,\\
\\
& = e^{-Q_n(t,\lambda)} \mathbb E_{X,Y} \left[ e^{\mathcal E_n(t,\lambda)}-1 \right]+ O\left( \frac{1}{\sqrt{n}}\right) .
\end{array}
\]
Since uniformly in $X,Y$ we have  $|\mathcal E_n(t,\lambda)| \leq||\lambda||_1^2$, using the fact that there exists a constant $C_{\lambda}$ such that $|e^{x}-1| \leq C_{\lambda} |x|$ for $|x | \leq ||\lambda||_1^2$, we have 
\[
\begin{array}{ll}
\left| \mathbb E_{X,Y} \left[ e^{\mathcal E_n(t,\lambda)}-1 \right] \right|  \leq \mathbb E_{X,Y} \left[ \left| e^{\mathcal E_n(t,\lambda)}-1 \right|\right]  \leq  C_{\lambda} \mathbb E [|\mathcal E_n(t,\lambda)|]  \\
\\
  \leq\displaystyle{ C_{\lambda} \sum_{i,j} |\lambda_i \lambda_j | \times \mathbb E_{X,Y} \left[ \left| \frac{1}{n} \sum_{k=1}^n \cos \left(k(X-Y)+ \frac{k(t_i-t_j)}{n}  \right)\right|\right].}
\end{array}
\]
An immediate application of Cauchy--Schwarz inequality then yields 
\[
\mathbb E_{X,Y} \left[ \left| \frac{1}{n} \sum_{k=1}^n \cos \left(k(X-Y)+ \frac{k(t_i-t_j)}{n}  \right)\right|\right] \leq  \sqrt{\frac{1}{2n}}.
\]
As a conclusion, we get that 
\[
\Delta_n = O\left( \frac{1}{\sqrt{n}}\right).
\]
Therefore, choosing a subsequence of the form $n^{\gamma}$ with $\gamma>2$, by Borel--Cantelli Lemma, we obtain that $\mathbb P$ almost surely, we have 
\[
\lim_{n \to +\infty} \Psi_{n^{\gamma}}(t,\lambda)  = e^{-Q(t,\lambda)/2}.
\]
Now for a generic integer $m \geq 1$, we can choose $n$ such that $n^\gamma < m \leq (n+1)^{\gamma}$ and then 
\[
\left| \Psi_{n^{\gamma}}(t,\lambda)-\Psi_{m}(t,\lambda)\right| \leq \mathbb E_X \left[ |U| +| V| + |W|\right],
\]
where 
\[
\begin{array}{ll}
U & :=\displaystyle{\left(\frac{1}{\sqrt{n^{\gamma}}} - \frac{1}{\sqrt{m}} \right)\left( \sum_{k=1}^{n^{\gamma}} a_k \alpha_{k,n^{ \gamma}}(X)+b_k \beta_{k,n^{\gamma}}(X)\right)},\\
\\
V & := \displaystyle{\frac{1}{\sqrt{m}} \left( \sum_{k=1+n^{\gamma}}^m a_k \alpha_{k,m}(X) +b_k \beta_{k,m}(X)\right)}, \\
\\
W & := \displaystyle{\frac{1}{\sqrt{m}} \left( \sum_{k=1}^{n^{\gamma}} a_k \left( \alpha_{k,n^{ \gamma}}(X)-\alpha_{k,m}(X)\right) +b_k \left(\beta_{k,n^{\gamma}}(X)- \beta_{k,m}(X)\right)\right)}.
\end{array}
\]
Applying Cauchy--Schwarz inequality, using the orthogonality relations \eqref{eq.ortho1} and \eqref{eq.ortho3} and the uniform upper bound \eqref{eq.bound}, we get 
\[
\mathbb E_X [|U|] \leq ||\lambda||_1 \left(1- \sqrt{\frac{n^{\gamma}}{m}}\right) \sqrt{ \frac{1}{n^{\gamma}} \sum_{k=1}^{n^{\gamma}} \frac{a_k^2 +b_k^2}{2} }.
\]
By the strong law of large number, the above square root is $\mathbb P$ almost surely bounded, hence we have $\mathbb P$ almost surely
\[
\mathbb E_X [|U|] \leq O \left(1- \sqrt{\frac{n^{\gamma}}{m}}\right)  = O\left(\frac{1}{n}\right)= O\left(\frac{1}{m^{1/\gamma}}\right).
\]
In the same manner, we have 
\[
\mathbb E_X [|V|] \leq ||\lambda||_1 \times \sqrt{ \frac{1}{m} \sum_{k=1+n^{\gamma}}^m \frac{a_k^2 +b_k^2}{2}} = O\left(\frac{1}{\sqrt{n}}\right)=O\left(\frac{1}{m^{1/2\gamma}}\right).
\]
Finally, using once again Cauchy--Schwarz inequality and the orthogonality relations \eqref{eq.ortho1} and \eqref{eq.ortho2}, since $\max_{1\leq i \leq p}|t_i| \leq 2\pi$, we have $\mathbb P$ almost surely 
\[
\mathbb E_X [|W|] \leq ||\lambda||_1 \times \sqrt{2\pi} \times \sqrt{1-\frac{n^{\gamma}}{m}}  \times \sqrt{ \frac{1}{m} \sum_{k=1}^{n^{\gamma}} (a_k^2 +b_k^2)}=O\left(\frac{1}{n}\right) = O\left(\frac{1}{m^{1/\gamma}}\right) .
\]
As a conclusion, we get that $\mathbb P$ almost surely 
\[
\lim_{m \to +\infty} \left| \Psi_{n^{\gamma}}(t,\lambda)-\Psi_{m}(t,\lambda)\right| =0,
\]
hence the result.
\end{proof}

Let us now establish the tightness of the family of distributions of $(g_n(t))_{t \in [0,2\pi]}$ under $\mathbb P_X$, for the $\mathcal C^1$ topology.
\begin{prop}\label{prop.tension}
Almost surely with respect to the probability $\mathbb P$, the family of distributions under $\mathbb P_X$ of $(g_n(t))_{t \in [0,2\pi]}$ for $n \geq1$ is tight with respect to the $\mathcal C^1$ topology on $\mathcal C^1([0,2\pi])$. 
\end{prop}

\begin{proof}By Theorem 1 and Remark 1 of \cite{rusakov}, in order to establish the tightness in the $\mathcal C^1$ topology of the family $(g_n(t))_{t \in [0,2\pi]}$, it is sufficient to establish some Lamperti-type criteria for both $\mathbb E_X\left[  |g_n(t)-g_n(s)|^2\right]$ and $\mathbb E_X\left[  |g_n'(t)-g_n'(s)|^2\right]$.
Using the trigonometric identities 
\[
\begin{array}{ll}
\cos \left(kX+ \frac{kt}{n}\right) -\cos\left(kX+ \frac{ks}{n}\right) & = -2 \sin \left(kX+ \frac{k(t+s)}{2n}\right)\sin \left(\frac{k(t-s)}{2n}\right),\\
\\
\sin \left(kX+ \frac{kt}{n}\right) -\sin\left(kX+ \frac{ks}{n}\right) & = 2 \cos \left(kX+ \frac{k(t+s)}{2n}\right)\sin \left(\frac{k(t-s)}{2n}\right),
\end{array}
\]
we have first, for all $1\leq k \leq n$
\[
\mathbb E_X\left[\left(\cos \left(kX+ \frac{kt}{n}\right) -\cos\left(kX+ \frac{ks}{n}\right)\right)^2\right] =2 \sin \left(\frac{k(t-s)}{2n}\right)^2,
\]
and 
\[
\mathbb E_X\left[\left(\sin \left(kX+ \frac{kt}{n}\right) -\sin\left(kX+ \frac{ks}{n}\right)\right)^2\right] =2 \sin \left(\frac{k(t-s)}{2n}\right)^2.
\]
By the strong law of large numbers, $C:=\sup_{n \geq 1} \frac{1}{2n} \sum_{k=1}^n (a_k^2+b_k^2) $ is $\mathbb P$ almost surely bounded. Therefore, using the fact that $\mathrm{vect}(\cos(kx), \sin(kx))$ and $\mathrm{vect}(\cos( \ell x), \sin( \ell x))$ are orthogonal in $\mathbb L^2([0,2\pi])$ for $1\leq k,\ell \leq n$ with $k \neq \ell$,  and that $\cos(kx) \perp \sin(kx)$, we have thus $\mathbb P$ almost surely, for all $s,t \in [0,2\pi]$
\[
\mathbb E_X\left[  |g_n(t)-g_n(s)|^2\right] = \frac{2}{n} \sum_{k=1}^n \left( a_k^2 + b_k^2\right) \;  \sin^2 \left(\frac{k(t-s)}{2n}\right) \leq C |t-s|^2.
\]
In the same manner, we get for all $s,t \in [0,2\pi]$
\[
\mathbb E_X\left[  |g_n'(t)-g_n'(s)|^2\right] =\sum_{k=1}^n \frac{2k^2}{n^3}\left( a_k^2 + b_k^2\right) \;  \sin^2 \left(\frac{k(t-s)}{2n}\right) \leq C |t-s|^2.
\]
hence the result.
\end{proof}

\begin{rem}\label{rem.bulin}
Note that the limit process $(g_{\infty}(t))_{t \in [0,2\pi]}$ is non degenerate in the sense that, $\mathbb P \otimes \mathbb P_X$ almost surely, we have 
\[
(g_{\infty}(t) = 0 ) \Longrightarrow (g_{\infty}'(t) \neq  0).
\]
Indeed, by stationarity we have $\mathbb E[g_{\infty}(t)^2]=1$, hence the classical Bulinskaya Lemma applies, see  e.g. Proposition 6.11 of \cite{azaisW}.
\end{rem}

\subsection{Salem--Zygmund Theorem for a non-uniform distribution}\label{sec.quali}
Before giving the proof of Theorem \ref{SZ-gene}, which generalizes Theorem \ref{gene-quali-theo} to a large class of distributions $\mathbb P_X$, let us first give an alternative proof of Theorem \ref{gene-quali-theo}.
\begin{proof}[Alternative proof of Theorem \ref{gene-quali-theo}]
The proof below is based on the averaging procedure raised in Section \ref{sec.introquali}. Namely we show that the $\mathbb L^2(\mathbb P)$ distance between the characteristic function of $f_n(X)$ under $\mathbb P_X$ and its limit goes sufficiently fast to zero in order to apply a Borel--Cantelli argument and conclude that the convergence is in fact almost sure under $\mathbb P$. To do so, let us set, for $\xi \in \mathbb R$
\[
\Delta_n(\xi):=\E\left[\left|\E_X\left[e^{i\xi f_n(X)}-e^{-\frac{\xi^2}{2}}\right]\right|^2\right].
\]
If $Y$ is an independent copy of $X$, by Fubini Theorem, one first write
\[
\begin{array}{ll}
\Delta_n(\xi) &=\E\left[\E_X\left[e^{i\xi f_n(X)}-e^{-\frac{\xi^2}{2}}\right]\E_Y\left[e^{-i\xi f_n(Y)}-e^{-\frac{\xi^2}{2}}\right]\right]\\
&=\mathbb{E}_{X,Y}\left[\mathbb{E}\left[e^{i\xi\left(f_n(X)-f_n(Y)\right)}+e^{-\xi^2}-e^{-\frac{\xi^2}{2}}\E_{X,Y}\left[e^{i\xi f_n(X)}+e^{-i\xi f_n(Y)}\right]\right]\right].
\end{array}
\]
Hence, we have $\Delta_n(\xi) = \mathbb{E}_{X,Y}\left[\Delta_{n,1}(\xi) \right]- e^{-\frac{\xi^2}{2}}  \mathbb{E}_{X,Y}\left[\Delta_{n,2}(\xi)\right]$ with 
\[
\begin{array}{ll}
\Delta_{n,1}(\xi) & :=\E\left[e^{i\xi \left(f_n(X)-f_n(Y)\right)}-e^{-\xi^2}\right]\\
\\
\Delta_{n,2}(\xi) & :=\E\left[e^{i\xi f_n(X)}-e^{-\frac{\xi^2}{2}}\right]+\E\left[e^{-i\xi f_n(Y)}-e^{-\frac{\xi^2}{2}}\right].
\end{array}
\]
The convergence to zero of $\Delta_{n,1}(\xi) $, resp. $\Delta_{n,2}(\xi)$, is then a simple application of the standard bivariate, resp. univariate, Central Limit Theorem under $\mathbb P$. Namely,
\[
\E\left[e^{i\xi \left(f_n(X)-f_n(Y)\right)}\right]=\prod_{k=1}^n \phi\left(\xi\frac{\cos(k X)-\cos( k Y)}{\sqrt{n}}\right)\phi\left(\xi\frac{\sin(k X)-\sin( k Y)}{\sqrt{n}}\right),
\]
where $\phi$ denotes the characteristic function of $a_1$. Since $a_1 \in \mathbb L^3\left(\mathbb{P}\right)$, the function $\log(\phi)$ is three times differentiable at the neighborhood of $0$. Using Taylor--Lagrange inequality and the fact that the random variables are centered with variance one, for some absolute constant $C>0$ which may change from line to line and for every $|t|<\frac{1}{2}$:
\begin{equation}\label{Esti-cara}
\left|\log\left(\phi(t)\right)+\frac{t^2}{2}\right|< C |t|^3.
\end{equation}
Let $\xi$ be fixed in the sequel, then for $n$ large enough we have
\[
\forall k\in\{1,\cdots,n\},\;\; \left|\xi\frac{\cos(k X)-\cos( k Y)}{\sqrt{n}}\right|<\frac 1 2.
\]
In the forthcoming computations, $n$ is assumed to be large enough so that the aforementioned estimate holds. Since \eqref{Esti-cara} applies, we obtain
\begin{eqnarray*}
&&\left|\sum_{k=1}^n \log\left(\phi\left(\xi\frac{\cos(k X)-\cos( k Y)}{\sqrt{n}}\right)\right)+\frac{1}{2}\sum_{k=1}^n \left(\xi\frac{\cos(k X)-\cos( k Y)}{\sqrt{n}}\right)^2\right|\\
&&<C \frac{|\xi^3|}{n\sqrt{n}}\sum_{k=1}^n \left|\cos(k X)-\cos( k Y)\right|^3<C\frac{|\xi|^3}{\sqrt{n}}.
\end{eqnarray*}
In the same way, we have
\begin{eqnarray*}
&&\left|\sum_{k=1}^n \log\left(\phi\left(\xi\frac{\sin(k X)-\sin( k Y)}{\sqrt{n}}\right)\right)+\frac{1}{2}\sum_{k=1}^n \left(\xi\frac{\sin(k X)-\sin( k Y)}{\sqrt{n}}\right)^2\right|\\
&&<C \frac{|\xi^3|}{n\sqrt{n}}\sum_{k=1}^n \left|\sin(k X)-\sin( k Y)\right|^3<C\frac{|\xi|^3}{\sqrt{n}}.
\end{eqnarray*}
Note that $\left(\cos(k X)-\cos( k Y)\right)^2+\left(\sin(k X)-\sin( k Y)\right)^2=2+2\cos\left(k(X-Y)\right)$ and set 
\begin{equation}\label{eq.def.epsilon}
\varepsilon_n=\varepsilon_n(X,Y):=\frac{1}{n}\sum_{k=1}^n \cos(k(X-Y)).
\end{equation}
Combining the two last estimates yields 
\[
\left|\sum_{k=1}^n \log\left(\phi\left(\xi\frac{\cos(k X)-\cos( k Y)}{\sqrt{n}}\right)\phi\left(\xi\frac{\sin(k X)-\sin( k Y)}{\sqrt{n}}\right)\right) +\xi^2 (1+ \varepsilon_n)\right|\le C\frac{|\xi|^3}{\sqrt{n}}.
\]
The exponential function being Lipschitz, one thus obtains for some constant $C(\xi)$, which changes from line to line and which only depends on $\xi$
\begin{eqnarray*}
&&\left|\prod_{k=1}^n \phi\left(\xi\frac{\cos(k X)-\cos( k Y)}{\sqrt{n}}\right)\phi\left(\xi\frac{\sin(k X)-\sin( k Y)}{\sqrt{n}}\right)- e^{-\xi^2 (1+ \varepsilon_n)}\right| \le  \frac{C(\xi)}{\sqrt{n}}.
\end{eqnarray*}
Hence, we have
\[
\left|\Delta_{n,1}(\xi)\right| \le \frac{C(\xi)}{\sqrt{n}}+\left| e^{-\xi^2 (1+ \varepsilon_n)} - e^{-\xi^2}\right| \leq \frac{C(\xi)}{\sqrt{n}} + \xi^2 \varepsilon_n
\]
and taking the expectation under $\mathbb P_X \otimes \mathbb P_Y$, by Cauchy--Schwarz inequality, we get\
\[
\left|\mathbb E_{X,Y} \left[ \Delta_{n,1}(\xi)\right] \right|\leq\frac{C(\xi)}{\sqrt{n}} +\xi^2 \underbrace{\sqrt{\E_{X,Y}\left[\varepsilon_n(X,Y)^2\right]}}_{=\sqrt{\frac{1}{2n}}\,\,\text{by ortho. of cosines}}=\frac{C(\xi)}{\sqrt{n}}.
\]
Following the exact same strategy, we have in the same manner
\[
\left|\mathbb E_{X,Y} \left[ \Delta_{n,2}(\xi)\right] \right|\leq \frac{C(\xi)}{\sqrt{n}}.
\]
Gathering these facts provide the estimate $\Delta_n(\xi) \le \frac{C(\xi)}{\sqrt{n}}$. Using Borel--Cantelli Lemma, one thus gets that $\mathbb{P}$ almost surely
\[ 
\E_X\left[e^{i\xi f_{n^3 }(X)}-e^{-\frac{\xi^2}{2}}\right] \xrightarrow[n \to +\infty]{} 0.
\] 
Applying the latter for each $\xi\in\mathbb{Q}$, and then $\xi \in \mathbb R$ by continuity, we have $\mathbb{P}$ almost surely, the central convergence of $f_{ n^3 }(X)$ under $\mathbb{P}_X$.
On the other hand, for every positive integers $m,n$ such that $n^3  \le m < (n+1)^3 $, we also have
\[
\begin{array}{ll}
\left|  \E_X\left[e^{i\xi f_{n^3 }(X)}\right] -\E_X\left[e^{i\xi f_{m }(X)}\right] \right| & \leq \xi \times \mathbb E_X \left[ \left|f_{n^3 }(X) - f_{m }(X)\right| \right]  \\
\\
& \leq \xi \sqrt{\E_X\left[\left(f_{n^3}(X)-f_m(X)\right)^2\right]},
\end{array}
\]
and as in the end of the proof of Theorem  \ref{Theo-Salem-quanti}, we have
\[
\begin{array}{ll}
\E_X\left[\left(f_{n^3}(X)-f_m(X)\right)^2\right] \\
\\
\displaystyle{\le 2 \left(\sqrt{\frac{n^3}{m}}-1\right)^2 \underbrace{\E_X\left[f_{n^3}^2(X)\right]}_{=1}+\frac{2}{n^3}\; \E_X\left[\left(\sum_{k=n^3+1}^m a_k\cos(k X)+b_k\sin(k X)\right)^2\right]}\\
\\
\le\displaystyle{ 2 \left(\sqrt{\frac{n^3}{(n+1)^3}}-1\right)^2+2\frac{(n+1)^3-n^3}{n^3} = O\left(\frac{1}{n}\right) = O\left(\frac{1}{m^{1/3}}\right).}
\end{array}
\]
As announced, $\mathbb{P}$ almost surely, $f_n(X)$ indeed converges in distribution to standard Gaussian variable under $\mathbb{P}_X$.
\end{proof}

Let us now give the proof of Theorem \ref{SZ-gene}, which roughly follows the same lines as the alternative proof of Theorem \ref{gene-quali-theo} given above, except that we shall skip some explicit computations using characteristic functions and replace them by the standard quantitative Berry--Esseen estimates in the Central Limit Theorem. For the sake of concision, we only focus on the main steps of the proof.

\begin{proof}[Proof of Theorem \ref{SZ-gene}]Let $X$ and $Y$ be independent, identically distributed random variables, fulfilling the assumptions of the statement. By the multivariate central limit Theorem, conditionally on $(X,Y)$, the couple $(f_n(X),f_n(Y))$ converges in distribution, under $\mathbb{P}$, towards a centered Gaussian vector $(G_X,G_Y)$ such that $\mathbb E_X[G_X^2]=\mathbb E_X[G_Y^2]=1$ and $\mathbb E_X[G_XG_Y]=\varepsilon_n(X,Y)$ given by \eqref{eq.def.epsilon}.
Then, the standard Berry--Esseen inequality implies that the speed of the aforementioned convergence is of order $\frac{1}{\sqrt{n}}$. Otherwise, a simple computation gives the formula 
\[
\cos(k(x-y))\cos(l(x-y))=\frac{1}{2}\left(\cos((k+l)(x-y))+\cos((k-l)(x-y))\right).
\]
Noticing that $\E_{X,Y}\left[e^{i k (X-Y)}\right]=\left|\hat{\mathbb{P}}(k)\right|^2$, one thus gets
\[
\E_{X,Y}\left[\varepsilon_n^2(X,Y)\right]=\frac{1}{2 n^2}\sum_{k,l=1}^n \left(\left|\hat{\mathbb{P}_X}(k+l)\right|^2+\left|\hat{\mathbb{P}_X}(k-l)\right|^2\right)
=O \left( \frac{1}{n^{2\alpha}}\right).
\]
Therefore, under $\mathbb{P}$, $\left(f_n(X),f_n(Y)\right)$ converges in distribution to a standard bivariate Gaussian vector at a polynomial speed $\frac{1}{n^c}$ with $c=\min\left(\alpha,\frac{1}{2}\right)$. Using again Borel--Cantelli Lemma, one deduces the central convergence under $\mathbb{P}_X$ for the subsequence $\lfloor n^a \rfloor$ with $a>\frac{1}{c}$. Finally, for $n^a<m<(n+1)^a$ one has
\begin{eqnarray*}
\E\left[\E_X\left[\left(f_m(X)-f_{\lfloor n^a \rfloor}(X)\right)^2\right]\right]=\E_X\left[\E\left[\left(f_m(X)-f_{\lfloor n^a \rfloor}(X)\right)^2\right]\right]\\
= O\left(  \left(\sqrt{\frac{(n+1)^a}{n^a}}-1\right)^2+\frac{(n+1)^a-n^a}{n^a}\right) = O\left( \frac{1}{n}\right) = O\left( \frac{1}{m^{\frac{1}{a}}}\right).
\end{eqnarray*}
Using hypercontractivity estimates for $\mathbb L^p$ norms of multi-linear polynomials, for $p>a$, one then deduces that
\[
\E\left[\E_X\left[\left(f_m(X)-f_{\lfloor n^a \rfloor}(X)\right)^{2p}\right]\right]= O\left( \frac{1}{m^{\frac{p}{a}}}\right).
\]
Hence, Borel--Cantelli Lemma again ensures that $\E_X\left[\left(f_m(X)-f_{\lfloor n^a \rfloor}(X)\right)^2\right]$ goes to 0 as $n$ goes to infinity, $\mathbb{P}$ almost surely. This concludes the proof, since under our assumptions, $\beta$ is large enough to find such an exponent $p$.
\end{proof}

\subsection{Convergence in total variation} \label{sec.tv}

We give here the proof of Theorem \ref{theo.dtv} stated in Section \ref{sec.statetv}, dealing with the total variation 
version of Theorem \ref{gene-quali-theo}. The proof is based on a variational calculus which is associated with a so-called Wright--Fisher differential operator. 

\begin{proof}[Proof of Theorem \ref{theo.dtv}]
Let us consider  a random variable $I$ which is uniformly distributed over $\{0,1,\cdots,n-1\}$ and which is independent of $X$. We denote by $\mathbb E_{I,X}$ the expectation with respect to the product lax $\mathbb P_I \otimes \mathbb P_X$. We start the proof by noticing that

\begin{equation}\label{change-law}
X_n:=\frac{2\pi}{n}I+\frac{X}{n}\stackrel{\text{law}}{=}X.
\end{equation}
Indeed, for every suitable function $\phi$ one may write
\[
\begin{array}{ll}
\displaystyle{\E_{I,X}\left[\phi\left(\frac{2\pi}{n}I+\frac{X}{n}\right)\right]} & =\displaystyle{\frac{1}{n}\sum_{k=0}^{n-1}\frac{1}{2\pi}\int_0^{2\pi}\phi\left(\frac{2\pi k}{n}+\frac{t}{n}\right) dt}\\
\\
&\displaystyle{=\sum_{k=0}^{n-1}\int_{\frac{2\pi k}{n}}^{2\pi\frac{k+1}{n}} \phi(u) du=\frac{1}{2\pi}\int_{0}^{2\pi}\phi(u) du=\E_X\left[\phi(X)\right]},
\end{array}
\]
where we have made the change of variables $u=\frac{2\pi k}{n}+\frac{t}{n}$. Let us now, introduce the following operators
$$
\left\{
\begin{array}{ll}
\forall (\phi,\psi)\in\mathcal{C}^1([0,2\pi],\R), \quad \Gamma[\phi,\psi]:=\phi'(x)\psi'(x)x\left(2\pi-x\right),\\
\\
\forall \psi \in \mathcal{C}^2([0,2\pi],\R), \quad \mathcal{L}\left[\psi\right]:=\psi''(x) (2\pi-x) x+\psi'(x) (2\pi-2 x).
\end{array}
\right.
$$
Then, one has the following integration by parts formula, which is the cornerstone of the proof, $\forall \phi\in\mathcal{C}^1([0,2\pi]),\,\forall\psi\in\mathcal{C}^2([0,2\pi]),\forall \theta \in\mathcal{C}^1([0,2\pi])$
\begin{equation}\label{IPP1}
\int_{0}^{2\pi} \Gamma[\phi,\psi]\theta(x) dx=-\int_{0}^{2\pi} \phi(x)\theta(x) \mathcal{L}\left[\psi\right](x) dx-\int_0^{2\pi} \phi(x) \Gamma\left[\theta(x),\psi(x)\right]dx.
\end{equation}
Let us now fix some continuous function $\phi$ which we assume to be compactly supported and bounded by one.  Let us also introduce its anti-derivative $\Phi(x)=\int_0^{x} \phi(t) dt$. One first write
\[
\E_X\left[\phi\left(f_n(X)\right)\right]=\E_{I,X}\left[\phi\left(f_n\left(\frac{2\pi}{n}I+\frac{X}{n}\right)\right)\right] =\mathbb{E}_{I,X}\left[\phi\left(g_{n,I}(X)\right)\right]\\
\]
with $g_{n,I}(x)=f_n\left(\frac{2\pi}{n}I+\frac{x}{n}\right)$. We then write, for $\alpha>0$ 
\[
\begin{array}{ll}
\displaystyle{\mathbb{E}_{I,X}\left[\phi\left(g_{n,I}(X)\right)\right]} & = \displaystyle{\mathbb{E}_{I,X}\left[\phi\left(g_{n,I}(X)\right)\frac{\Gamma[g_{n,I}(X),g_{n,I}(X)]+\alpha}{\Gamma[g_{n,I}(X),g_{n,I}(X)]+\alpha}\right]}\\
\\
& \displaystyle{\le\underbrace{\mathbb{E}_{I,X}\left[\phi\left(g_{n,I}(X)\right)\frac{\Gamma[g_{n,I}(X),g_{n,I}(X)]
}{\Gamma[g_{n,I}(X),g_{n,I}(X)]+\alpha}\right]}_{:=A_n(\phi)}}\\
&\displaystyle{+\underbrace{\mathbb{E}_{I,X}\left[\frac{\alpha}{\Gamma[g_{n,I}(X),g_{n,I}(X)]+\alpha}\right]}_{:=B_n(\alpha)}}
\end{array}
\]
with $\Gamma[g_{n,I}(X),g_{n,I}(X)]= g_{n,I}'(X)^2\left(2\pi-X\right)X$.
Otherwise, using the chain rule associated with $\Gamma$, we have $\phi\left(g_{n,I}(X)\right)\Gamma[g_{n,I}(X),g_{n,I}(X)]=\Gamma\left[\Phi\left(g_{n,I}(X)\right),g_{n,I}(X)\right]$, and the integration by parts formula \eqref{IPP1} leads to
\[
\begin{array}{ll}
 \displaystyle{A_n(\phi)} & = \displaystyle{-\mathbb{E}_{I,X}\left[\Phi\left(g_{n,I}(X)\right) \frac{\mathcal{L}\left[g_{n,I}(X)\right]}{\Gamma[g_{n,I}(X),g_{n,I}(X)]+\alpha}\right]}\\
\\
& \displaystyle{-\mathbb{E}_{I,X}\left[\Phi\left(g_{n,I}(X)\right) \Gamma\left[g_{n,I}(X),\frac{1}{\Gamma[g_{n,I}(X),g_{n,I}(X)]+\alpha}\right]\right]}.
\end{array}
\]
We then notice that
\[
\left|\mathbb{E}_{I,X}\left[ \frac{\Phi\left(g_{n,I}(X)\right) \times \mathcal{L}\left[g_{n,I}(X)\right]}{\Gamma[g_{n,I}(X),g_{n,I}(X)]+\alpha}\right]\right| \leq  \frac{\|\Phi\|_\infty}{\alpha} \mathbb{E}_{I,X}\left[\left|\mathcal{L}\left[g_{n,I}(X)\right]\right|\right]\le C \frac{\|\Phi\|_\infty}{\alpha},
\]
with $C=4\pi \sup_n \sqrt{\frac{1}{n}\sum_{k=1}^n (a_k^2+b_k^2)}\stackrel{\text{a.s.}}{<}\infty$.
Indeed, relying on the Cauchy--Schwarz inequality and the orthogonality of the family  $(\cos(k X),\sin(k X))_{k\ge 1}$ we have
\[
\begin{array}{ll}
\mathbb{E}_{I,X}\left[|2\pi-X| \left|g_{n,I}''(X)\right|\right] & \le 2\pi \mathbb{E}_{I,X}\left[\left|g_{n,I}''(X)\right|\right] 
\le  2\pi \sqrt{\frac{1}{n}\sum_{k=1}^n\left(a_k^2+b_k^2\right)}\\
\\
\mathbb{E}_{I,X}\left[|2\pi-2 X| \left|g_{n,I}'(X) \right|\right] &\le 2\pi \sqrt{\frac{1}{n}\sum_{k=1}^n\left(a_k^2+b_k^2\right)}.
\end{array}
\]
Using the chain-rule property and the definition of the operator $\Gamma$, we first write
\[
\left|\Gamma\left[[g_{n,I}(X),\frac{1}{\Gamma[g_{n,I}(X),g_{n,I}(X)]+\alpha}\right]\right|=\left|\frac{\Gamma[g_{n,I}(X),\Gamma[g_{n,I}(X),g_{n,I}(X)]]}{\left(\Gamma[g_{n,I}(X),g_{n,I}(X)]+\alpha\right)^2}\right|
\]
with 
\[
\begin{array}{ll}
\Gamma[g_{n,I}(X),\Gamma[g_{n,I}(X),g_{n,I}(X)]] & =2 g_{n,I}'(X)^2 g_{n,I}''(X) (2\pi-X)X \\
\\
& +g_{n,I}'(X)^3 (2\pi-X)X (2\pi-2X).
\end{array}
\]
Since $|g_{n,I}'(X)^2 (2\pi-X)X|=\Gamma[g_{n,I}(X),g_{n,I}(X)]\le \Gamma[g_{n,I}(X),g_{n,I}(X)]+\alpha$,
we get
\[
\left|\Gamma\left[[g_{n,I}(X),\frac{1}{\Gamma[g_{n,I}(X),g_{n,I}(X)]+\alpha}\right]\right|\leq  \frac{\left|g_{n,I}''(X)\right|+2\pi |g_{n,I}'(X)|}{\alpha}.
\]
We then deduce, following the same previous reasoning that
\begin{eqnarray*}
\left|\mathbb{E}_{I,X}\left[\Phi\left(g_{n,I}(X)\right) \Gamma\left[g_{n,I}(X),\frac{1}{\Gamma[g_{n,I}(X),g_{n,I}(X)]+\alpha}\right]\right]\right|\le C\frac{\|\Phi\|_\infty}{\alpha}.
\end{eqnarray*}
In conclusion we have proved that $\mathbb{P}$ almost surely, there is a constant $C>0$ such that, for every integer $n\ge1$, $|A_n(\phi)|\le C \frac{\|\Phi\|_\infty}{\alpha}.$ One also have, for all $\lambda>0$
\[
\begin{array}{ll}
|B_n(\alpha)|& = \displaystyle{\E_X\left[\frac{\alpha}{\alpha+|g_{n,I}'(X)|^2(2\pi-X)X}\right] }\\
\\
&\displaystyle{ \le\E_X\left[\frac{\alpha}{\alpha+|g_{n,I}'(X)|^2 \lambda}\right]+\mathbb{P}_X\left(X(2\pi-X)<\lambda\right)}.
\end{array}
\]
Using Salem--Zygmund CLT, we know that 
\[
g_{n,I}'(X) \stackrel[\textrm{under}\; \mathbb P_X]{\textrm{law}}{=} \displaystyle{\frac{1}{\sqrt{n}}\sum_{k=1}^n\frac{k}{n} \left(-a_k  \sin\left(X\right)+b_k \cos\left(X\right)\right)}
\]
converges in distribution towards a non-degenerate Gaussian distribution $G$. This ensures us that
\begin{eqnarray*}
&&\limsup_{n\to\infty} |B_n(\alpha)| \le \mathbb{P}_X\left((2\pi-X)X \le \lambda\right)+\E_X\left[\frac{\alpha}{\alpha+G^2\lambda}\right].
\end{eqnarray*}
Letting $\alpha\to 0$ and then $\lambda\to 0$ one obtains that
\begin{equation}\label{ineq-avant-lemme}
\lim_{\alpha\to0}\limsup_{n\to\infty} |B_n(\alpha)|=0.
\end{equation}
If $\mu_n$ is the law of $g_{n,I}(X)$ under $\mathbb P_{I} \otimes \mathbb P_X$, which naturally coincides with the law of $f_n(X)$ under $\mathbb P_X$, and if $\mu_\infty$ is the law of the standard Gaussian, the desired conclusion proceeds from the following lemma, whose proof is given in Section \ref{sec.lemmatv} of the Appendix. 
\begin{lem}\label{gene-lemma-TV}
Let $\mu_n$ a sequence of probability distributions which converges in law towards $\mu_\infty$. Assume further that for any continuous function bounded by one and compactly supported $\phi$ we have
\[
\left| \int_\R \phi(x) d\mu_n(x) \right| \le \frac{\sup_x\left|\int_0^x \phi(t) dt\right|}{\alpha}+|B_n(\alpha)|\,\,\quad\&\quad\,\,\lim_{\alpha\to 0}\limsup_n \left|B_n(\alpha)\right|=0.
\]
Then $\mu_n$ converges towards $\mu_\infty$ in total variation.
\end{lem}
\end{proof}

\section{Almost sure asymptotics for the number of real zeros} \label{sec.applizero}
This whole section is devoted to the proof of Theorem \ref{theo.as} stated in Section \ref{sec.statezero}.
Let us recall that the object of interest is the random trigonometric polynomial
\[
f_n(t) :=\frac{1}{\sqrt{n}} \sum_{k=1}^n a_k \cos(kt)+b_k \sin(kt), \quad t \in \mathbb R,
\]
and that the number of real zeros of $f_n$ in a given interval $[a,b]$ is denoted by
\[
\mathcal N(f_n, [a,b]) := \# \{ t \in [a,b], \; f_n(t)=0\}.
\]

\subsection{A probabilistic representation of the number of zeros} \label{sec.formula}
Let us start with the very simple following result.

\begin{lem}\label{lem.start}
If $f$ is a $2\pi-$periodic function with a finite number of zeros, then for any $0<h<2\pi$, we have  
\[
\frac{h}{2\pi} \times \mathcal N(f, [0,2\pi]) = \mathbb E_X \left[  \mathcal N\left(f, \left[X,X+h\right]\right)\right],
\]
where $X$ is a random variable, with uniform distribution in $[0, 2\pi]$.
\end{lem}

\begin{proof}
Set $N=\mathcal N(f, [0,2\pi])$ which is finite by hypothesis, and denote by $x_1, \ldots, x_N$ the zeros of $f$ in $[0,2\pi]$ and $\mu_f$ the associated empirical measure 
\[
\mu_f := \frac{1}{N} \sum_{k=1}^N \delta_{x_k}.
\]
Naturally, we have for all $a < b$ such that $b-a \leq 2\pi$
\[
\mathcal N(f, [a,b]) = N \times \int_0^{2\pi} \mathds{1}_{[a,b] \, \mathrm{mod} \,2\pi}(t) \mu_f(dt).
\]
If $X$ is uniform in $[0,2\pi]$, we have then applying Fubini inversion of sums
\[
\begin{array}{ll}
\mathbb E_X \left[  \mathcal N\left(f, \left[X,X+h\right]\right) \right] & =\displaystyle{ \frac{1}{2\pi} \int_{0}^{2\pi} \mathcal N\left(f, \left[x,x+h\right]\right) dx} \\
\\
& =\displaystyle{ \frac{N}{2\pi} \int_{0}^{2\pi} \left( \int_{0}^{2\pi}  \mathds{1}_{[x,x+h] \, \mathrm{mod} \,2\pi}(t) dx\right)\mu_f(dt)}\\
\\
& =\displaystyle{ \frac{N}{2\pi} \int_{0}^{2\pi} \left( \int_{0}^{2\pi}  \mathds{1}_{[t-h,t]\, \mathrm{mod} \,2\pi}(x) dx\right)\mu_f(dt)}\\
\\
& = \displaystyle{ \frac{N}{2\pi} \times h \times \int_{0}^{2\pi}\mu_f(dt) = \frac{N}{2\pi} \times h.}
\end{array}
\]
\end{proof}
\noindent
Applying Lemma \ref{lem.start} with $f_n$ and $h=2\pi/n$, we thus get if $X$ is a random variable with uniform distribution on $[0,2\pi]$, independent of the sequences $(a_k)_{k \geq 1}$ and $(b_k)_{k \geq 1}$, then 
\begin{equation}\label{eq.start}
\frac{\mathcal N(f_n, [0,2\pi])}{n} = \mathbb E_X \left[  \mathcal N\left(f_n, \left[X,X+\frac{2\pi}{n}\right]\right)\right].
\end{equation}
\subsection{Functional convergence in distribution}
\label{sec.TCLfunc}
The last Equation \eqref{eq.start} naturally invites to consider the process $(g_n(t))_{t \in [0,2\pi]}$ defined by 
\[
g_n(t):=f_n\left(X+ \frac{t}{n}\right), \quad t \in [0,2\pi],
\]
so that 
\begin{equation}\label{eq.start2}
\frac{\mathcal N(f_n, [0,2\pi])}{n} =  \mathbb E_X \left[\mathcal N\left(g_n, \left[0,2\pi\right]\right)\right].
\end{equation}
By Theorem \ref{theo.functional}, we know that $(g_n(t))_{t \in [0,2\pi]}$ converges in distribution in the $\mathcal C^1$ topology to a limit Gaussian process $(g_{\infty}(t))_{t \in [0,2\pi]}$ with $\sin_c$ covariance function, which is almost surely non-degenerate by Remark \ref{rem.bulin}. Furthermore, it is well known that the number of zeros of a smooth function in a given interval is a continuous functional for the $\mathcal C^1$ topology, in the neighborhood of any non-degenerate function, see \cite{rusakov,nousAMS}. Therefore, by the continuous mapping Theorem, we get the following asymptotics. 

\begin{prop}\label{prop.conv}
Almost surely with respect to $\mathbb P$, as $n$ goes to infinity 
\[
\mathcal N\left(f_n, \left[X,X+\frac{2\pi}{n}\right]\right) = \mathcal N\left(g_n, \left[0,2\pi\right]\right)
\]
converges in distribution under $\mathbb P_X$ to $\mathcal N\left(g_{\infty}, \left[0,2\pi\right]\right)$.
\end{prop}

\subsection{Logarithmic integrability} \label{sec.logint}

Having the above Proposition \ref{prop.conv} in mind, the proof of Theorem \ref{theo.as} would be complete if we had some moment control on the number of zeros of $g_n$ in the interval $[0,2\pi]$. As a first step towards this goal, let us first establish some logarithmic moment estimates for $f_n(X)=g_n(0)$.

\begin{lem}Suppose that the random variables $(a_k,b_k)_{k \geq 1}$ are independent standard Gaussian variables or more generally are that they are independent and identically distributed with a symmetric distribution with bounded variance. Then $\mathbb P$ almost surely, for all $p>1$ and for all $0< \theta<1$, there exists a constant $C_{p,\theta}=C_{p,\theta}(\omega)$ such that
\begin{equation}\label{eq.croisspoly}
\mathbb E_X \left[ \left| \log( |f_n(X)|)\right|^p\right]  \leq C_{p,\theta} \times n^{\theta}.
\end{equation}
\label{lem.croisspoly}
\end{lem}

\begin{proof}
We first give the proof in the Gaussian framework, in which case it is elementary. Indeed, if $(a_k,b_k)_{k \geq 1}$ are independent standard normal variables, then under $\mathbb P$, the variable $f_n(X)$ is then also a standard Gaussian variable so that for all $p>1$
\[
\mathbb E \left[ \mathbb E_X \left[ \left| \log( |f_n(X)|)\right|^p\right]\right] = \int_\mathbb R \left| \log( |x|)\right|^p \frac{e^{-\frac{x^2}{2}}}{\sqrt{2\pi}}dx=:\kappa(p) <+\infty.
\]
Using Markov and Jensen inequalities, we have thus 
\[
\begin{array}{ll}
\displaystyle{\mathbb P \left(  \mathbb E_X \left[ \left| \log( |f_n(X)|)\right|^p\right] > n^{\theta}\right)}  & = \displaystyle{\mathbb P \left(  \mathbb E_X \left[ \left| \log( |f_n(X)|)\right|^p\right]^{\frac{2}{\theta}} > n^{2}\right)}\\
\\
 & \leq \displaystyle{\frac{1}{n^2} \mathbb E \left[ \mathbb E_X \left[ \left| \log( |f_n(X)|)\right|^p\right]^{\frac{2}{\theta}} \right]} \\
 \\
 & \leq \displaystyle{\frac{1}{n^2} \mathbb E \left[ \mathbb E_X \left[ \left| \log( |f_n(X)|)\right|^{\frac{2p}{\theta}}\right] \right] \leq \frac{\kappa(\frac{2p}{\theta})}{n^2}}.
\end{array}
\]
Therefore, by Borel--Cantelli Lemma, we deduce that $\mathbb P$ almost surely, for $n$ sufficiently large, we have 
$\mathbb E_X \left[ \left| \log( |f_n(X)|)\right|^p\right]  \leq  n^{\theta}$, hence the result.
Let us now turn to the more general case of symmetric random variables. 
The proof then follows the same lines as in the Gaussian case but the starting point, i.e. the finiteness of the moments under $\mathbb P \otimes \mathbb P_X$ is here ensured by the powerful results of \cite{nishry} for Rademacher Fourier series. Indeed, let us introduce another probability space $(\Omega_{\varepsilon,\eta}, \mathcal F_{\varepsilon,\eta}, \mathbb P_{\varepsilon,\eta})$ which carries a sequence $(\varepsilon_k, \eta_k)_{k \geq 1}$ of independent symmetric Rademacher variables. Then, under $\mathbb P$, the whole sequence $(a_k,b_k)_{k \geq 1}$ has the same law as the sequence $(\varepsilon_k |a_k|, \eta_k |b_k|)_{k \geq 1}$ under $\mathbb P \otimes \mathbb P_{\varepsilon, \eta}$. Following Corollary 1.2 of \cite{nishry}, for any $p >1$, there exists a deterministic constant $K_p$ such that
\[
\mathbb E_{\varepsilon, \eta} \left[ \mathbb E_X  \left[ \left| \log\left( \frac{\left|\sum_{k=1}^n \varepsilon_k |a_k| \cos(k X) +\eta_k |b_k| \sin(k X) \right|}{\sqrt{\sum_{k=1}^n \frac{a_k^2+b_k^2}{2}}}\right)\right|^p\right] \right] \leq K_{p} .
\]
In particular, we get
\[
\mathbb E_{\varepsilon, \eta} \left[ \mathbb E_X  \left[ \left| \log\left( \left|\frac{1}{\sqrt{n}} \sum_{k=1}^n \varepsilon_k |a_k| \cos(k X) +\eta_k |b_k| \sin(k X) \right|\right)\right|^p\right] \right] \leq K_{p,n},
\]
where we have set 
\[
K_{p,n}=K_{p,n}(\omega):=2^p K_p + \left| \log \left( \frac{1}{n}\sum_{k=1}^n \frac{a_k^2+b_k^2}{2}\right)\right|^p.
\]
By Markov inequality with respect to $\mathbb P_{\varepsilon,\eta}$, as in the proof of Lemma \ref{lem.croisspoly}, we have then for all $\theta>0$
\[
\mathbb P_{\varepsilon,\eta} \left( \mathbb E_X  \left[ \left| \log\left( \left|\frac{1}{\sqrt{n}} \sum_{k=1}^n \varepsilon_k |a_k| \cos(k X) +\eta_k |b_k| \sin(k X) \right|\right)\right|^p\right] > n^{\theta}\right) \leq \frac{K_{\frac{2p}{\theta},n}}{n^2}.
\]
Now, by the law of large numbers, $\mathbb P$ almost surely, for $n \geq n_0=n_0(\omega)$ sufficiently large, we have $K_{p,n} \leq 2^p K_p+1$ so that applying Borel--Cantelli  Lemma with respect to $\mathbb P_{\varepsilon, \eta}$, we thus get that $\mathbb P \otimes \mathbb P_{\varepsilon, \eta}$ almost surely, for $n$ sufficiently large 
\[
 \mathbb E_X  \left[ \left| \log\left( \left|\frac{1}{\sqrt{n}} \sum_{k=1}^n \varepsilon_k |a_k| \cos(k X) +\eta_k |b_k| \sin(k X) \right|\right)\right|^p\right] \leq n^{\theta}.
 \]
But since the law of $(\varepsilon_k |a_k|, \eta_k |b_k|)_{k \geq 1}$ under $ \mathbb P \otimes \mathbb P_{\varepsilon, \eta}$ coincides with the one of $(a_k,b_k)_{k \geq 1}$ under $\mathbb P$, we can thus conclude that $\mathbb P$ almost surely, for $n$ sufficiently large 
\[
 \mathbb E_X  \left[ \left| \log\left( \left|\frac{1}{\sqrt{n}} \sum_{k=1}^n a_k \cos(k X) +b_k \sin(k X) \right|\right)\right|^p\right] \leq n^{\theta},
 \]
hence the result.
\end{proof}

Thanks to the explicit rate of convergence in Salem--Zygmund estimates established in Theorem \ref{Theo-Salem-quanti} and more precisely in Corollary \ref{speed-cv-ps} above, the last result can actually be reinforced to show that the logarithmic moments are in fact $\mathbb P$ almost surely bounded. 

\begin{lem}Suppose that the random variables $(a_k,b_k)_{k \geq 1}$ are independent standard Gaussian variables or more generally are that they are independent and identically distributed with a symmetric distribution with bounded variance.  Then $\mathbb P$ almost surely, for all $p>1$, there exists a constant $C_{p}=C_{p}(\omega)$ such that uniformly in $n \geq 1$
\[
\mathbb E_X \left[ \left| \log( |f_n(X)|)\right|^p\right]  \leq C_{p}.
\]\label{lem.borne}
\end{lem}

\begin{proof}[Proof of Lemma \ref{lem.borne}]
Let $a>0$, then for $x$ small enough, we have 
\[
|\log(a+x)| = \left \lbrace \begin{array}{ll} |\log(a) | -P_a(x)+o(x^3), & \mathrm{if} \;\; 0<a<1, \\
\\
 |\log(a) | +P_a(x) +o(x^3), & \mathrm{if} \;\; a>1,
\end{array} \right.
\]
where $P_a(x):=\frac{x}{a} - \frac{x^2}{2a^2} + \frac{x^3}{3a^3} $.
Fix $M>e$ so that $\log(M)>1$ and let us introduce the following smooth truncation of the logarithm function, defined for all $x \geq 0$ 
\[
\log_M(x):=\left \lbrace \begin{array}{ll}  
\log(M)-(Mx)^4 \times P_{\frac{1}{M}}\left(x-\frac{1}{M}\right),& \mathrm{if} \;\; x \leq \frac{1}{M},\\
\\
|\log(x)| & \mathrm{if} \;\; \frac{1}{M}<x<M,\\
\\
 \log(M)+(1-(x-M))^4 \times P_{M}(x-M), & \mathrm{if} \;\; M\leq x \leq M+1,\\
 \\
 \log(M), & \mathrm{if} \;\; x \geq M+1.
\end{array} \right.
\]
The function $x \mapsto \log_M(x)$ is then non-negative and bounded by $\log(M)+2$ on $\mathbb R^+$ and it satisfies $\log_M(x) \leq |\log(x)|$. It coincides with $x \mapsto |\log(x)|$ on the interval $[1/M,M]$, with matching derivatives up to order three at $x=1/M$ and $x=M$, and vanishing derivatives up to order three at $x=0$ and $x=M+1$. In particular, it is $\mathcal C^3$ on each of the intervals $[0,1)$ and $(1,+\infty)$. Moreover, in the neighborhood of $1$, we have $\log_M(1+x)=|\log(1+x)| \approx |x|$. As a result, for all $p>3$, the function $x \in \mathbb R \mapsto \log_M(|x|)^p$ is $C^3$ on $\mathbb R$, and we have 
\[
\max\left( ||\log_M^p||_{\infty}, ||{\log_M^p}'||_{\infty},||{\log_M^p}''||_{\infty}, ||{\log_M^p}^{(3)}||_{\infty}\right) \leq p^3 M^3\log(M)^{p-1}.
\]
 By the triangle inequality, if $G \sim \mathcal N(0,1)$ under $\mathbb P_X$, we can now write 
 \[
 \begin{array}{ll}
 \mathbb E_X \left[ \left| \log( |f_n(X)|)\right|^p\right] & \leq  \underbrace{\left| \mathbb E_X \left[ \left| \log( |f_n(X)|)\right|^p  \right]  -  \mathbb E_X \left[ \left| \log_M( |f_n(X)|)\right|^p\right]\right|}_{A_n}   \\
 \\
& + \underbrace{\left|  \mathbb E_X \left[ \left| \log_M( |f_n(X)|)\right|^p  \right]  -  \mathbb E_X \left[ \left| \log_M( |G|)\right|^p\right] \right|}_{B_n}\\
\\
 & +  \mathbb E_X \left[ \left| \log_M( |G|)\right|^p\right].
 \end{array}
 \]
 Since $\log_M(|x|) \leq |\log(|x|)|$ on $\mathbb R$, we have $\mathbb E_X \left[ \left| \log_M( |G|)\right|^p\right] \leq \mathbb E_X \left[ \left| \log( |G|)\right|^p\right]\leq \kappa_p.$
Moreover, thanks to Theorem \ref{Theo-Salem-quanti} and Corollary \ref{speed-cv-ps}, since the function $\log_M^p$ is of class $\mathcal C^3$ with bounded derivatives, for some $\beta<1/6$, we have $\mathbb P$ almost surely for $n$ large enough 
\[
B_n \leq p^3 M^3\log(M)^{p-1} \times d_{\mathcal C^3}^X( f_n(X), G) \leq \frac{C \, p^3 M^3\log(M)^{p-1}}{n^{\beta}}.
\]
Finally, we can upper bound $A_n$ in the following way
\[
A_n  \leq  A_{n,1} +A_{n,2}+A_{n,3}+A_{n,4},
\]
where 
\[
\begin{array}{ll}
 A_{n,1} & =  \mathbb E_X \left[ \left| \log( |f_n(X)|)\right|^p \mathds{1}_{|f_n(X) |<1/M}\right]   \\
 \\
 A_{n,2} &:= \mathbb E_X \left[ \left| \log_M( |f_n(X)|)\right|^p \mathds{1}_{|f_n(X) |<1/M}  \right]\\
\\
 A_{n,3} & =  \mathbb E_X \left[ \left| \log( |f_n(X)|)\right|^p \mathds{1}_{|f_n(X) |>M}\right]   \\
 \\
 A_{n,4} &:= \mathbb E_X \left[ \left| \log_M( |f_n(X)|)\right|^p \mathds{1}_{|f_n(X) |>M}  \right].\\
\end{array}
\]
Next, since $||\log_M||_{\infty} \leq \log(M)+2$, we have 
\[
\begin{array}{ll}
A_{n,2} & \leq (\log(M)+2)^p \times \mathbb P_X \left( |f_n(X) |<1/M\right), \\
\\
A_{n,4} & \leq (\log(M)+2)^p \times \mathbb P_X \left( |f_n(X) |>M\right). \\
\end{array}
\]
By Markov inequality, we deduce that 
\[
A_{n,4}  \leq (\log(M)+2)^p \times  \frac{\mathbb E_X \left[|f_n(X) |^2\right]}{M^2} = \frac{(\log(M)+2)^p}{M^2} \times \frac{1}{n}\sum_{k=1}^n \frac{a_k^2+b_k^2}{2}.
\]
Let us now make use of the following general lemma which allow to compare the $\mathcal C^3$ distance used in Theorem \ref{Theo-Salem-quanti} and Corollary \ref{speed-cv-ps} to the more classical Kolmogorov distance. Its proof is given in Section \ref{sec.lemmaKol} of the appendix.
\begin{lem}\label{lem.compar}
Let $(Y_n)$ a sequence of random variables which converges in distribution towards a standard Gaussian variable $G$ under $\mathbb P_X$, then we have 
\[
\sup_{x \in \mathbb R}\left| \mathbb P_X( G \leq x) - \mathbb P_X( Y_n \leq x) \right| = O \left( d_{C^3}^{X}(Y_n,G)^{\frac{1}{4}}\right).
\]
\end{lem}
\noindent
Applying Lemma \ref{lem.compar} with $Y_n=f_n(X)$, in conjunction with Corollary \ref{speed-cv-ps}, we have
\[
\left| \mathbb P_X \left( |f_n(X) |<\frac{1}{M}\right) -  \mathbb P_X \left( |G|<\frac{1}{M}\right) \right| = O \left( d_{\mathcal C^3}^X\left( f_n(X),G \right)^{\frac{1}{4}}\right) = O\left(\frac{1}{n^{\frac{\beta}{4}}}\right).
\]
As a result, we get 
\[
\begin{array}{ll}
A_{n,2}  & \leq (\log(M)+2)^p \times \left| \mathbb P_X \left( |G |<\frac{1}{M}\right)+ O\left(\frac{1}{n^{\frac{\beta}{4}}}\right)\right] \\
\\
& \leq (\log(M)+2)^p \times \left[ O\left(\frac{1}{M} \right)+ O\left(\frac{1}{n^{\frac{\beta}{4}}}\right)\right].
\end{array}
\]
Otherwise, using Cauchy--Schwarz inequality, we have 
\[
\begin{array}{ll}
| A_{n,1}|^2  & \leq  \mathbb E_X \left[ \left| \log( |f_n(X)|)\right|^{2p} \right] \mathbb P_X \left(|f_n(X) |<1/M \right),\\
 \\
 | A_{n,3}|^2  & \leq  \mathbb E_X \left[ \left| \log( |f_n(X)|)\right|^{2p} \right] \mathbb P_X \left(|f_n(X) |>M \right).
 \end{array}
\]
By Lemma \ref{lem.croisspoly}, we then get that $\mathbb P$ almost surely, for all $\theta>0$
\[
|A_{n,3}|^2   \leq  \mathbb E_X \left[ \left| \log( |f_n(X)|)\right|^{2p} \right] \frac{\mathbb E_X \left[|f_n(X) |^2\right]}{M^2}  \leq C_{2p,\theta} \times n^{\theta}\times \frac{1}{M^2} \times \frac{1}{n}\sum_{k=1}^n \frac{a_k^2+b_k^2}{2},
\]
and as above
\[
|A_{n,1}|^2  \leq C_{2p,\theta} \times n^{\theta}\times \left[ O\left(\frac{1}{M} \right)+ O\left(\frac{1}{n^{\frac{\beta}{4}}}\right)\right].
\]
As a conclusion, choosing $M=M(n)$ of the form $n^{\gamma}$ with $0<\gamma<\beta/4$, and $\theta< \gamma$, we obtain that $\mathbb P$ almost surely, as $n$ goes to infinity
\[
\mathbb E_X \left[ \left| \log( |f_n(X)|)\right|^p\right]  \leq \kappa_p + o(1), 
\]
hence the result.

\end{proof}

\subsection{Moment estimates} \label{sec.moments}

Let us now describe how the above estimates of Section \ref{sec.logint} on the logarithmic moments of $|f_n(X)|=|g_n(0)|$ actually allow to obtain some moment estimates for the number of zeros $\mathcal N(g_n,[0,2\pi])$ of $g_n$ in $[0,2\pi]$. Precisely, the goal of this section is to prove the following result. 

\begin{prop}\label{pro.moment}
For all $p\geq 1$, $\mathbb P$ almost surely we have 
\[
\sup_{n \geq 1} \mathbb E_X \left[ |\mathcal N(g_n,[0,2\pi])|^p\right]< +\infty.
\]
\end{prop}

\begin{proof}
Let us fix $p \geq 1$. We first write
\[
\mathbb E_X \left[ |\mathcal N(g_n,[0,2\pi])|^p\right]  =p \int_0^{+\infty} s^{p-1} \mathbb P_X \left(\mathcal N(g_n,[0,2\pi])>s  \right)ds.
\]
By iterating Rolle Lemma $\lfloor s \rfloor$ times, see for example p.19 of \cite{nousAMS}, the last probability can then be upper bounded as follows
\[
\mathbb P_X \left(\mathcal N(g_n,[0,2\pi])>s  \right)  \leq \mathbb P_X \left(|g_n(0)| \leq \frac{(2\pi)^{\lfloor s \rfloor}}{\lfloor s \rfloor !}||g_n^{(\lfloor s \rfloor)}||_{\infty}\right),
\]
so that for any $R>0$
\begin{equation}\label{eq.bornprob0}
\mathbb P_X \left(\mathcal N(g_n,[0,2\pi])>s  \right)  \leq \mathbb P_X \left(|g_n(0)| \leq \frac{(2\pi)^{\lfloor s \rfloor}R}{\lfloor s \rfloor !} \right)+ \mathbb P_X \left( ||g_n^{(\lfloor s \rfloor)}||_{\infty}>R\right).
\end{equation}
Applying Markov inequality, we get
\[
\mathbb P_X \left( ||g_n^{(\lfloor s \rfloor)}||_{\infty}>R\right) \leq \frac{1}{R^2} \times \mathbb E_X \left[ ||g_n^{(\lfloor s \rfloor)}||_{\infty}^2 \right].
\]
Now, comparing the uniform norm with Sobolev norms, see Lemma 5.15 p.107 of \cite{adams}, if $|| \cdot ||_{2}$ denotes the standard $\mathbb L^2$ norm in $\mathbb L^2([0,2\pi])$, there exists a universal constant $C>0$ such that 
\[
\mathbb E_X \left[ ||g_n^{(\lfloor s \rfloor)}||_{\infty}^2 \right]  \leq C \left(\mathbb E_X \left[   ||g_n^{(\lfloor s \rfloor)}||_{2}^2\right] +\mathbb E_X \left[   ||g_n^{(\lfloor s \rfloor+1)}||_{2}^2 \right] \right).
\]
But for any integer $\ell$, we have 
\[
\mathbb E_X \left[   ||g_n^{( \ell )}||_{2}^2\right] = \frac{1}{2n} \sum_{k=1}^n \left(\frac{k}{n}\right)^{2\ell} \left( a_k^2+b_k^2\right) \leq \frac{1}{2n} \sum_{k=1}^n \left( a_k^2+b_k^2\right), 
\]
hence 
\[
\mathbb P_X \left( ||g_n^{(\lfloor s \rfloor)}||_{\infty}>R\right) \leq  \frac{2C}{R^2} \times \frac{1}{2n} \sum_{k=1}^n \left( a_k^2+b_k^2\right).
\]
Let us now choose $R=R(s)$ of the form $R(s):=(1+|s|)^p$ to get that 
\begin{equation}\label{eq.bornprob1}
\sup_{n \geq 1}\mathbb P_X \left( ||g_n^{(\lfloor s \rfloor)}||_{\infty}>(1+|s|)^p\right) \leq \frac{2C}{(1+|s|)^{2p}} \times \sup_{n \geq 1} \frac{1}{2n} \sum_{k=1}^n \left( a_k^2+b_k^2\right).
\end{equation}
Now, for $s$ large enough we have then $\frac{(2\pi)^{\lfloor s \rfloor}R(s)}{\lfloor s \rfloor !}<1$, and by Lemma \ref{lem.borne}, uniformly in $n$ and for all  $q \geq 1$, we have $\mathbb P$ almost surely 
\[
\begin{array}{ll}
\mathbb P_X \left(|g_n(0)| \leq \frac{(2\pi)^{\lfloor s \rfloor}R(s)}{\lfloor s \rfloor !} \right) & = \mathbb P_X \left(|f_n(X)| \leq \frac{(2\pi)^{\lfloor s \rfloor}R(s)}{\lfloor s \rfloor !} \right) \\
\\
&  = \mathbb P_X \left(\left| \log \left( |f_n(X)| \right)\right|  \geq \left| \log\left( \frac{(2\pi)^{\lfloor s \rfloor}R(s)}{\lfloor s \rfloor !} \right) \right|\right)\\
\\
& \leq \displaystyle{ \frac{\mathbb E_X \left[\left| \log \left( |f_n(X)| \right)\right|^q \right]}{\left| \log\left( \frac{(2\pi)^{\lfloor s \rfloor}R(s)}{\lfloor s \rfloor !} \right) \right|^q} \leq  \frac{C_q}{\left| \log\left( \frac{(2\pi)^{\lfloor s \rfloor}R(s)}{\lfloor s \rfloor !} \right) \right|^q}}, 
\end{array}
\]
that is, for all  $q \geq 1$, $\mathbb P$ almost surely, for $s$ large enough
\begin{equation}\label{eq.bornprob2}
\sup_{n \geq 1} \mathbb P_X \left(|g_n(0)| \leq \frac{(2\pi)^{\lfloor s \rfloor}R(s)}{\lfloor s \rfloor !} \right)= O \left( \left| \frac{1}{\lfloor s \rfloor \log \left \lfloor s \rfloor\right)}\right|^q \right).
\end{equation}
Choosing $q>p$ and combining Equations \eqref{eq.bornprob0}, \eqref{eq.bornprob1} and \eqref{eq.bornprob2}, we thus get that $\mathbb P$ almost surely 
 \[
 \sup_{n \geq 1}\int_0^{+\infty} s^{p-1} \mathbb P_X \left(\mathcal N(g_n,[0,2\pi])>s  \right)ds <+\infty,
 \]
 hence the result.
\end{proof}

Combining the last Proposition \ref{pro.moment} with the convergence in distribution established in Proposition \ref{prop.conv}, we obtain the convergence of all the moments. 
\begin{coro}\label{cor.convmoment}
For all $p \geq 1$, we have $\mathbb P$ almost surely
\[
\lim_{n \to +\infty} \mathbb E_X\left[  \mathcal N(g_n,[0,2\pi])^p\right] = \mathbb E_X\left[  \mathcal N(g_{\infty},[0,2\pi])^p\right].
\]
\end{coro}

\subsection{Conclusion}\label{sec.conclu}
Following Lemma \ref{lem.start} and Equations \eqref{eq.start} and \eqref{eq.start2}  , we have 
\[
\frac{\mathcal N(f_n, [0,2\pi])}{n} = \mathbb E_X \left[  \mathcal N\left(g_n, \left[0,2\pi \right]\right)\right].
\]
Applying the above Corollary \ref{cor.convmoment}, with $p=1$, we thus get that $\mathbb P$ almost surely 
\[
\lim_{n \to +\infty} \frac{\mathcal N(f_n, [0,2\pi])}{n}  = \mathbb E_X\left[  \mathcal N(g_{\infty},[0,2\pi])\right].
\]
By Theorem \ref{theo.functional}, under $\mathbb P_X$, the limit process $(g_{\infty}(t))_{t \in [0,2\pi]}$ is a Gaussian process with $\sin_c$ covariance so that, as it is well known
\[
\mathbb E_X\left[  \mathcal N(g_{\infty},[0,2\pi])\right] = \frac{2}{\sqrt{3}}.
\]
More generally, a similar proof would yield that $\mathbb P$ almost surely, for any $[a,b] \subset [0,2\pi]$, 
\[
\lim_{n \to +\infty} \frac{\mathcal N(f_n, [a,b])}{n}  = \frac{(b-a)}{\pi\sqrt{3}} .
\]

The above almost sure convergence can be rephrased as follows in terms of the empirical measure associated with the real zeros of $f_n$.

\begin{coro}
Let us denote by $\nu_n$ the empirical measure associated with the real roots of the random trigonometric polynomial $f_n$, namely
\[
\nu_n := \frac{1}{\mathcal N(f_n,[0,2\pi])}\sum_{ \substack{x \in [0,2\pi[ \\ f_n(x)=0}} \delta_x.
\]
Then, $\mathbb P$ almost surely, as $n$ goes to infinity, $\nu_n$ converge in distribution to the normalized Lebesgue measure on $[0,2\pi]$.  
\end{coro}

\newpage
\appendix
\section{Appendix}

\subsection{Proof of Lemma \ref{lem.Antilde}}\label{app.lem1}

We give here the proof of Lemma \ref{lem.Antilde}, which is the key step in the proof of Theorem \ref{Theo-Salem-quanti}, i.e. the quantified version of Salem--Zygmund convergence. 

\begin{proof}[Proof of Lemma \ref{lem.Antilde}] We have the following decomposition
\begin{eqnarray*}
\tilde{A}_n(\xi)&:=&\mathbb{E}\left(\left|\frac{1}{\sqrt{n}}\sum_{k=1}^n\mathbb{E}_X\left[R_k(X)\exp\left(i \xi  S_n^{k}(X)\right)\right]\right|^2\right)\\
&=&\E_X\E_Y\left[\mathbb{E}\left[\frac{1}{n}\sum_{k,l=1}^n R_k(X) R_l(Y) \exp\left(i \xi  \left(S_n^{k}(X)-S_n^{l}(Y)\right)\right)\right]\right]\\
&=&\underbrace{\E_X\E_Y\left[\mathbb{E}\left[\frac{1}{n}\sum_{k=1}^n R_k(X) R_k(Y) \exp\left(i \xi  \left(S_n^{k}(X)-S_n^{k}(Y)\right)\right)\right]\right]}_{\text{diagonal term:=Diag}}\\
&+&\underbrace{\E_X\E_Y\left[\mathbb{E}\left[\frac{1}{n}\sum_{k\neq l} R_k(X) R_l(Y) \exp\left(i \xi  \left(S_n^{k}(X)-S_n^{l}(Y)\right)\right)\right]\right]}_{\text{off diagonal term:=Off-Diag}}.
\end{eqnarray*}

\medskip
\noindent
\underline{\textit{The diagonal term:}}\par
\medskip
\noindent
We start by handling the diagonal term. We have
\[
\begin{array}{ll}
\mathrm{Diag} & =\displaystyle{\E_X\E_Y\left[\frac{1}{n}\sum_{k=1}^n \mathbb{E}\left[R_k(X) R_k(Y)\right] \mathbb{E}\left[\exp\left(i \xi  \left(S_n^{k}(X)-S_n^{k}(Y)\right)\right)\right]\right]}\\
\\
& = \displaystyle{\E_X\E_Y\left[\mathbb{E}\left[\frac{1}{n}\sum_{k=1}^n \cos(k(X-Y)) \exp\left(i \xi  \left(S_n^{k}(X)-S_n^{k}(Y)\right)\right)\right]\right]}\\
\\
& = \displaystyle{\underbrace{\E_X\E_Y\left[\mathbb{E}\left[\frac{1}{n}\sum_{k=1}^n \cos(k(X-Y)) \exp\left(i \xi  \left(S_n(X)-S_n(Y)\right)\right)\right]\right]}_{:=\mathrm{Diag}_1}}\\
\\
& + \displaystyle{\underbrace{\E_X\E_Y\left[\frac{1}{n}\sum_{k=1}^n \cos(k(X-Y)) \mathbb{E}\left[ e^{i \xi  \left(S_n^{k}(X)-S_n^{k}(Y)\right)}\right] \mathbb{E}\left[1-\Psi_{k,k,X,Y}(\xi)\right]\right]}_{:=\mathrm{Diag}_2}},
\end{array}
\]
where we have set, for $1\leq k, l \leq n$
\[
\Psi_{k,l,X,Y}(\xi):=\exp\left(i\frac{\xi}{\sqrt{n}}\left(R_k(X)-R_l(Y)\right)\right).
\]
A Taylor expansion at the order $2$ gives the following uniform bound:
\[
\left|\Psi_{k,k,X,Y}(\xi)-1-i\frac{\xi}{\sqrt{n}}\left(R_k(X)-R_k(Y)\right)\right| \le \frac{|\xi|^2}{2 n}.
\]
Since $R_k(X)-R_k(Y)$ is centered with respect to $\mathbb P$, we have  
\[
\left| \mathbb{E}\left[1-\Psi_{k,k,X,Y}(\xi)\right] \right| \leq \frac{|\xi|^2}{2 n}, 
\]
so that $\mathrm{Diag}_2 \leq \frac{\xi^2}{2 n}$. Moreover, using Fubini inversion, we have
$
\mathrm{Diag}_1= \E\left[\widetilde{\mathrm{Diag}}_1 \right]
$
where
\[
\begin{array}{ccl}
\widetilde{\mathrm{Diag}}_1 &:=& \displaystyle{\mathbb{E}_X\E_Y\left[\frac{1}{n}\sum_{k=1}^n \cos(k(X-Y)) \exp\left(i \xi  \left(S_n(X)-S_n(Y)\right)\right)\right]}\\
\\
&=&\displaystyle{\E_{X,Y}\left[\frac{1}{2 n} \sum_{k=1}^n \exp\left(i k (X-Y)\right) \exp\left(i \xi  \left(S_n(X)-S_n(Y)\right)\right)\right]}\\
\\
&+&\displaystyle{\E_{X,Y}\left[\frac{1}{2 n} \sum_{k=1}^n \exp\left(-i k (X-Y)\right) \exp\left(i \xi  \left(S_n(X)-S_n(Y)\right)\right)\right]}\\
\\
&=& \displaystyle{\frac{1}{2n}\sum_{k=1}^n \left|\frac{1}{2\pi}\int_0^{2\pi} e^{i k x} e^{i \xi S_n(x)} dx\right|^2}\displaystyle{+\frac{1}{2n}\sum_{k=1}^n \left|\frac{1}{2\pi}\int_0^{2\pi} e^{-i k x} e^{i \xi S_n(x)}  dx\right|^2}\\
\\
&\stackrel[\mathrm{inequality}]{\mathrm{Bessel}}{\leq}& \displaystyle{\frac{1}{2n}\left(\frac{1}{2\pi}\int_0^{2\pi} \left|\exp\left(i \xi S_n(x)\right) \right|^2 dx \right)=\frac{1}{2n}} .
\end{array}
\]
Gathering the above estimates, one thus gets
\begin{equation}\label{diago-final}
\left|\text{Diag}\right|\le \frac{|\xi|^2+1}{n}.
\end{equation}
\underline{\textit{The off-diagonal term:}}\par
\medskip
Let us define $\displaystyle{S_n^{k,l}(X):=S_n(X)-\frac{R_k(X)}{\sqrt{n}}-\frac{R_l(X)}{\sqrt{n}}}$, so that

\begin{equation}\label{eq-off-diago}
\textrm{Off-Diag}
=\mathbb{E}_{X,Y}\left[\frac{1}{n} \sum_{k\neq l} \E\left[e^{ i \xi  \left(S_n^{k,l}(X)-S_n^{k,l}(Y)\right)}\right]\mathbb{E}\left[R_k(X)R_l(Y)\Psi_{k,l,X,Y}(\xi)\right]\right].
\end{equation}
Performing this time a Taylor expansion at the order $3$ we get:
\[
\begin{array}{rl}
\Psi_{k,l,X,Y}(\xi) & =\displaystyle{
1+i\xi\frac{R_l(X)-R_k(Y)}{\sqrt{n}}-\frac{\xi^2}{2n}\left(R_l(X)-R_k(Y)\right)^2} \\
& \displaystyle{-\frac{i\xi^3}{n\sqrt{n}}\left(R_l(X)-R_k(Y)\right)^3+\mathcal{R}_{n,k,l,\xi},} \quad \textrm{with} \quad \displaystyle{\left|\mathcal{R}_{n,k,l,\xi}\right|\le \frac{|\xi|^4}{24 n^2}}.
\end{array}
\]
Now, since $k\neq l$, we have $\mathbb{E}\left[R_k(X)R_l(Y)\right] =0$ and in we introduce the following notation to lighten the expressions
$\Delta_{k,l}^p =\Delta_{k,l}^p(X,Y):= \mathbb{E}\left[R_k(X)R_l(Y)(R_l(X)-R_k(Y))^p\right]$ for $p=1,2,3$, we have then $\Delta_{k,l}^1=0$ and 
\[
\begin{array}{ll}
\Delta_{k,l}^2 & =-2\mathbb{E}\left[R_k(X)R_k(Y)\right]\mathbb{E}\left[R_l(X)R_l(Y)\right]=-2\cos(k(X-Y))\cos(l(X-Y)),\\
\\
\Delta_{k,l}^3 & =\underbrace{\mathbb{E}\left[R_k(X)R_l(Y)R_l^3(X)\right]}_{=0} -\underbrace{\mathbb{E}\left[R_k(X)R_k^3(Y)R_l(Y)\right]}_{=0}\\ 
\\&  -3\mathbb{E}\left[R_k(X)R_k(Y)R_l^2(X)R_l(Y)\right]+3\mathbb{E}\left[R_k(X)R_k^2(Y)R_l(Y)R_l(X)\right]\\
\\
& =-3 \mathbb E[a_1^3] \cos(k(X-Y))\left(\cos^2(l X)\cos(l Y)+\sin^2(l X)\sin( l Y)\right)\\
\\
& +3\mathbb E[a_1^3] \cos(l(X-Y))\left(\cos^2(k Y)\cos(k X)+\sin^2(k Y)\sin( k X)\right).
\end{array}
\]
Hence, replacing $\Psi_{k,l,X,Y}(\xi)$ by its Taylor expansion in \eqref{eq-off-diago}, we get
\[
\begin{array}{ll}
\text{Off-Diag} & =-\underbrace{\frac{\xi^2}{n^2}\E_{X,Y}\left[\sum_{k\neq l} \E\left[e^{i\xi \left(S_n^{k,l}(X)-S_n^{k,l}(Y)\right)}\right]\cos(k(X-Y))\cos(l(X-Y))\right]}_{\text{Off-diag}1}\\
&-\underbrace{\frac{i \xi^3}{n^2\sqrt{n}}\E_{X,Y}\left[\sum_{k\neq l}\E\left[e^{ i \xi  \left(S_n^{k,l}(X)-S_n^{k,l}(Y)\right)}\right] \Delta_{k,l}^3(X,Y)\right]}_{\text{Off-diag}2}+\mathcal{R}_{k,l,X,Y,\xi},
\end{array}
\]
where $\left|\mathcal{R}_{k,l,X,Y,\xi}\right|\leq \frac{\xi^4}{24 n}$.
Using a Taylor expansion at the order $1$ gives
\begin{eqnarray*}
&&\left|\E\left[e^{i\xi \left(S_n^{k,l}(X)-S_n^{k,l}(Y)\right)}\right]-\E\left[e^{i\xi \left(S_n(X)-S_n(Y)\right)}\right]\right|\\
&& \le\left|1-\mathbb{E}\left[e^{i \frac{\xi}{\sqrt{n}} \left(R_k(X)+R_l(X)-R_k(Y)-R_l(Y)\right)}\right]\right|\le\frac{\xi^2}{2n}\\
\end{eqnarray*}
Plugging this in the first off-diagonal term leads to
\begin{eqnarray*}
&&\left|\text{Off-diag}1-\frac{\xi^2}{n^2}\widetilde{\text{Off-diag}}1\right|\le \frac{\xi^4}{2 n},
\end{eqnarray*}
where
\[
\widetilde{\text{Off-diag}}1  :=\E_{X,Y}\left[\E\left[e^{i\xi \left(S_n(X)-S_n(Y)\right)}\right]\sum_{k\neq l}\cos(k(X-Y))\cos(l(X-Y))\right]
\]
satisfies
\[
\begin{array}{ll}
\left|\widetilde{\text{Off-diag}}1 \right| & \displaystyle{\le \E_{X,Y}\left[\left|\sum_{k\neq l}\cos(k(X-Y))\cos(l(X-Y))\right|\right]}\\
&\le \displaystyle{\E_{X,Y}\left[\left|\sum_{k=1}^n\sum_{l=1}^n\cos(k(X-Y))\cos(l(X-Y))\right|\right]+n}\\
&=\displaystyle{\E_{X,Y}\left[\left(\sum_{k=1}^n\cos(k(X-Y))\right)^2\right]+n=\frac{3}{2}n.}
\end{array}
\]
Finally, gathering all these facts leads to
\begin{equation}\label{off-diago1}
\left|\text{Off-diag1}\right|\le \frac{3 \xi^2}{2 n}+\frac{\xi^4}{2 n}.
\end{equation}
Next, by the same strategy since
$$\left|\E\left[e^{i\xi \left(S_n^{k,l}(X)-S_n^{k,l}(Y)\right)}\right]-\E\left[e^{i\xi \left(S_n(X)-S_n(Y)\right)}\right]\right|\le 4\frac{|\xi|}{\sqrt{n}},$$
one gets
\begin{eqnarray*}
\left|\text{Off-diag}2\right|&\le&12\left|\mathbb E[a_1^3]\right|\frac{\xi^4}{n}+\frac{|\xi|^3}{n^2\sqrt{n}}\mathbb{E}_{X,Y}\left[\left|\sum_{k\neq l} \Delta_{k,l}^3(X,Y)\right|\right].
\end{eqnarray*}
Let us simply notice that
\begin{eqnarray*}
&&\left|\sum_{k\neq l} \Delta_{k,l}^3(X,Y)-\sum_{k=1}^n\sum_{l=1}^n \Delta_{k,l}^3(X,Y)\right|\le 6 |\mathbb E[a_1^3]| n\\
&&\left|\sum_{k=1}^n\sum_{l=1}^n \Delta_{k,l}^3(X,Y)\right|\le 6 |\mathbb E[a_1^3]| \left|\sum_{k=1}^n \cos(k(X-Y))\right|.
\end{eqnarray*}
The latter ensures that
\begin{eqnarray*}
\mathbb{E}_{X,Y}\left[\left|\sum_{k\neq l} \Delta_{k,l}^3(X,Y)\right|\right]&\le& \mathbb{E}_{X,Y}\left[\left|\sum_{k=1}^n\sum_{l=1}^n \Delta_{k,l}^3(X,Y)\right|\right]+ 6 |\mathbb E[a_1^3]| n\\
&\le& 6 n |\mathbb E[a_1^3]| \times \mathbb{E}_{X,Y}\left[\left|\sum_{k=1}^n \cos (k (X-Y))\right|\right]+6|\mathbb E[a_1^3]| n\\
&\le& 12 |\mathbb E[a_1^3]| n \sqrt{n},
\end{eqnarray*}
where the Cauchy--Schwarz inequality as well as the orthogonality of the random variables $(\cos(k(X-Y)))_{k\ge1}$ have been used in the last inequality. As a matter of fact, one deduces that
\begin{equation}\label{off-diago2}
\left|\text{Off-diag}2\right|\le \frac{12 |\mathbb E[a_1^3]|}{n}\left(|\xi|^4+|\xi|^3\right).
\end{equation}
Finally, gathering the estimates \eqref{off-diago1} and \eqref{off-diago2} gives the following bound for the total off-diagonal term:
\begin{eqnarray}\label{off-diago-final}
\nonumber
\left|\text{Off-diag}\right|&\le& \frac{|\xi^4|}{24 n}+\frac{3 \xi^2}{2 n}+\frac{\xi^4}{2 n}+\frac{12 |\mathbb E[a_1^3]|}{n}\left(|\xi|^4+|\xi|^3\right)\\
&\le& \left(12+|\mathbb E[a_1^3]|\right)\frac{|\xi|^4+|\xi|^3+|\xi|^2}{n}.
\end{eqnarray}
Finally using both \eqref{diago-final} and \eqref{off-diago-final}, one indeed obtains the desired bound \eqref{Antilde-final}, namely
\[
\forall \xi \in \mathbb{R},\,\, \left|\tilde{A}_n(\xi)\right|\le|\text{Diag}|+|\text{Off-diag}|\le \left(13+|\mathbb E[a_1^3]|\right)\frac{|\xi|^4+|\xi|^3+|\xi|^2+1}{n}.
\]

\end{proof}

\subsection{Proof of Lemma \ref{gene-lemma-TV}} \label{sec.lemmatv}
Let us now give the proof of Lemma \ref{gene-lemma-TV} used in the proof of the total variation version of Salem--Zygmund convergence.

\begin{proof}[Proof of Lemma \ref{gene-lemma-TV}]
Let $\displaystyle{\frac{1}{\epsilon}\rho\left(\frac{\cdot}{\epsilon}\right)}$ be a regularization kernel, which we may assume compactly supported. Let $\phi$ be any compactly supported, continuous and bounded by $1$. We start by writing that
\begin{eqnarray*}
&&\left|\int_\R \phi(x) d\mu_n-\int_\R \phi(x)d\mu_\infty\right|\le \left|\int_\R \left(\phi(x)-\phi\ast\rho_\epsilon(x)\right) d\mu_n\right|\\
&&+\left|\int_\R \left(\phi(x)-\phi\ast\rho_\epsilon(x)\right) d\mu_\infty\right|
+\left|\int_\R \phi\ast\rho_\epsilon(x) d\mu_n-\int_\R \phi\ast\rho_\epsilon(x) d\mu_\infty\right|.
\end{eqnarray*}
By the assumptions on the sequence of probability measures $(\mu_n)_{n\ge1}$, we may infer that, for any $\alpha>0$,
\begin{eqnarray*}
&&\left|\int_\R \left(\phi(x)-\phi\ast\rho_\epsilon(x)\right) d\mu_n\right|\le \frac{\sup_{x\in\R} \left|\int_0^x \left( \phi(x)-\phi\ast\rho_\epsilon(x)\right) d\mu_n\right|}{\alpha}+B_n(\alpha).\end{eqnarray*}
Using Fubini inversion of sums one may write
\begin{eqnarray*}
\int_\R \left(\phi(x)-\phi\ast\rho_\epsilon(x)\right) d\mu_n&=&\int_0^x \phi(t)\int_{\R}\left(\rho_\epsilon(y)-\phi(t-y)\rho_\epsilon(y)\right)dydx\\
&=&\int_\R \rho_\epsilon(y)\left(\int_0^x\left(\phi(t)-\phi(t-y)\right)dt\right)dy\\
&=&\int_\R \rho(y) \left(\int_0^x\left(\phi(t)-\phi(t-\epsilon y)\right)dt\right)dy.
\end{eqnarray*}
Since $\|\phi\|_\infty\le 1$, one gets
\begin{eqnarray*}
\left|\int_\R \left(\phi(x)-\phi\ast\rho_\epsilon(x)\right) d\mu_n\right|&\le& \int_ \R \rho(y)\left(\int_{-\epsilon y}^{0}|\phi(t)| dt+\int_{x-\epsilon y}^{x}|\phi(t)| dt\right)+|B_n(\alpha)|
\\
&&\le \epsilon \int_\R |y| \rho(y) dy+|B_n(\alpha)|.
\end{eqnarray*}
Letting $n\to\infty$ in the previous inequality, using that $\mu_n$ converges in distribution towards $\mu_\infty$ and the assumptions on $B_n(\alpha)$, it also holds that
\[
\int_\R \phi(x) d\mu_n(x) \le \frac{\epsilon \int_\R |y| \rho(y) dy}{\alpha}.
\]
Let us recall that the convergence in distribution is a topology which can be associated with the so-called \textit{Fortet-Mourier} metric. This metric is defined as follows:
\[
d_{FM}(\mu,\nu) =\sup_{\substack{\|\phi\|_\infty\le 1\\\|\phi'\|_\infty\le 1}} \left|\int_\R \phi(x) d\mu(x)-\int_\R \psi(x) d\nu (x)\right|.
\]
Besides, if $\|\phi\|_\infty\le 1$ one deduces that $\|\phi\ast\rho_\epsilon\|_\infty\le 1$ and $\|\phi\ast\rho_\epsilon'\|_\infty\le \frac{\int_\R|\rho'(t)|dt}{\epsilon}.$
Gathering all these facts leads  to
\[
\begin{array}{ll}
\displaystyle{\left|\int_\R \phi(x) (d\mu_n(x)-d\mu_\infty(x)) \right| }& \displaystyle{\le 2 \frac{\epsilon}{\alpha} \int_\R |y|\rho(y) dy+|B_n(\alpha)|} \\
\\
& \displaystyle{+\left(1+\frac{\int_\R|\rho'(t)|dt}{\epsilon}\right)d_{FM}\left(\mu_n,\mu_\infty\right).}
\end{array}
\]
Since by assumption $\mu_n$ tends to $\mu_\infty$ in distribution and since $d_{FM}$ is a metric for this convergence, $d_{FM}\left(\mu_n,\mu_\infty\right)$ goes to zero as $n$ goes to infinity. Taking the $\limsup$ for $n$ going to infinity in the previous inequality gives 
\[
\limsup_{n \to +\infty} d_{TV}\left(\mu_n,\mu_\infty\right)\le \limsup_{n} |B_n(\alpha)|+2 \frac{\epsilon}{\alpha} \int_\R |y|\rho(y) dy
\]
Then letting  $\epsilon$ go to zero and next $\alpha$ go to zero, we get the desired conclusion.
\end{proof}

\subsection{Proof of Lemma \ref{lem.compar}}
\label{sec.lemmaKol}
Finally, let us give the proof of Lemma \ref{lem.compar} allowing to compare the Kolomogorov and $\mathcal C^3$ distances, which was used in the proof of Lemma \ref{lem.borne} on the finiteness of the logarithmic moments of $f_n(X)$.

\begin{proof}[Proof of Lemma \ref{lem.compar}]
Let us fix some small $h>0$ and introduce the smooth approximations of $\mathds{1}_{\cdot \leq x}$
\[
\rho_x^{h-}(y) := \left \lbrace \begin{array}{ll}  1 \;\; \mathrm{if} \;\; y \leq x-h, \quad 0 \;\; \mathrm{if} \;\; y \geq x \\
\\
1-\frac{1}{1+e^{-\frac{h}{(y-x)(y-x+h)}}} \;\; \mathrm{if} \;\; x -h \leq y \leq x,
\end{array} \right.
\]
and 
\[
\rho_x^{h+}(y) := \left \lbrace \begin{array}{ll}  1 \;\; \mathrm{if} \;\; y \leq x, \quad 0 \;\;  \mathrm{if} \;\; y \geq x+h \\
\\
1-\frac{1}{1+e^{\frac{h}{(y-x)(y-x-h)}}} \;\; \mathrm{if} \;\; x  \leq y \leq x+h. 
\end{array} \right.
\]
We have then $\rho_x^{h-}(y)  \leq \mathds{1}_{y \leq x} \leq \rho_x^{h+}(y)$ for all $y \in \mathbb R$, and there exists a constant $C$ such that 
\[
\max( ||\rho_x^{h\pm}||_{\infty}, ||{\rho_x^{h\pm}}'||_{\infty}, ||{\rho_x^{h\pm}}''||_{\infty}, ||{\rho_x^{h\pm}}'''||_{\infty}) \leq \frac{C}{h^3}.
\]
We can then write
\[
\begin{array}{lll}
\mathbb P_X( G \leq x) - \mathbb P_X( X_n \leq x)   &= & \mathbb E_X \left[ \mathds{1}_{G \leq x} - \rho_x^{h-}(G)\right] +\mathbb E_X \left[ \rho_x^{h-}(G) -\rho_x^{h-}(Y_n) \right]\\
\\ & & +\underbrace{\mathbb E_X\left[\rho_x^{h-}(Y_n)] -  \mathds{1}_{ Y_n \leq x} \right]}_{ \cdot \leq 0}\\
\\
& \leq & \mathbb P_X( G \in [x-h,x]) + \frac{C}{h^3} \times d_{C^3}^{X}(Y_n,G). 
\end{array}
\]
Hence, since the Gaussian density is uniformly bounded by $1/\sqrt{2\pi} \leq 1$, uniformly in $x$, we get 
\[ 
\mathbb P_X( G \leq x) - \mathbb P_X( X_n \leq x) \leq h + \frac{C}{h^3} \times d_{C^3}^{X}(Y_n,G).
\]
In the same manner, we have 
\[
\begin{array}{ll}
\mathbb P_X( X_n \leq x) - \mathbb P_X( G \leq x)   = & \underbrace{\mathbb E_X \left[ \mathds{1}_{X_n \leq x} - \rho_x^{h+}(X_n)\right]}_{ \cdot \leq 0} +\mathbb E_X \left[ \rho_x^{h+}(X_n) -\rho_x^{h+}(G) \right]\\
\\ & +\mathbb E_X\left[\rho_x^{h+}(G)] -  \mathds{1}_{ G \leq x} \right]\\
\\
& \leq \mathbb P( G \in [x,x+h]) + \frac{C}{h^3} \times d_{C^3}^{X}(Y_n,G). 
\end{array}
\]
Therefore, we have 
\[
\sup_{x \in \mathbb R}\left| \mathbb P_X( G \leq x) - \mathbb P_X( Y_n \leq x) \right| \leq h + \frac{C}{h^3} \times d_{C^3}^{X}(Y_n,G),
\]
which yields the desired result after optimizing in $h$.

\end{proof}

\noindent
{\bf Acknowledgments}: the authors would like to thank Igor Wigman for bringing to their attention the papers \cite{bourgain,buckley} which motivated this research.

\bibliographystyle{alpha}
\newcommand{\etalchar}[1]{$^{#1}$}

\end{document}